\documentclass[11pt]{article}
\usepackage{amsmath,amssymb,amsthm}
\usepackage{algpseudocode,algorithm,algorithmicx}
\usepackage{booktabs}
\usepackage{multirow}
\usepackage{braket}
\usepackage[colorlinks=true]{hyperref}
\usepackage[margin=0.9in]{geometry}
\newtheorem{theorem}{Theorem}
\newtheorem{proposition}[theorem]{Proposition}
\newtheorem{lemma}[theorem]{Lemma}

\newtheorem{remark}{Remark}
\newtheorem{definition}{Definition}
\newtheorem{assumption}{Assumption}
\DeclareMathOperator{\Prox}{Prox}
\DeclareMathOperator{\Proj}{Proj}
\DeclareMathOperator{\gph}{gph}
\DeclareMathOperator*{\argmin}{argmin}

\newcommand{\vp}{\varphi}
\newcommand{\R}{\mathbb{R}}
\newcommand{\Rb}{\overline{\mathbb{R}}}
\newcommand{\Mb}{m_{\mathrm{upper}}}
\newcommand{\mb}{m_{\mathrm{lower}}}

\begin{document}
\title{Coderivative-Based Newton Methods with Wolfe Linesearch\\ for Nonsmooth Optimization}
\author{
	Miantao Chao\footnote{College of Mathematics and Information Science, Guangxi University, Nanning, 530004, China. E-mail: chaomiantao@126.com. Research of this author was supported by National Science Foundation of China 12061013.} 
\and
Boris S. Mordukhovich\footnote{Department of Mathematics, Wayne State University, Detroit, Michigan, USA. E-mail: aa1086@wayne.edu. Research of this author was partly supported by the US National Science Foundation under grant DMS-2204519, by the Australian Research Council under Discovery Project DP-190100555, and by Project~111 of China under grant D21024.}
\and
Zijian Shi\footnote{College of Mathematics and Information Science, Guangxi University, Nanning, 530004, China. E-mail: zijianshi@foxmail.com.} 
\and 
Jin Zhang\footnote{Department of Mathematics, Southern University of Science and Technology, National Center for Applied Mathematics Shenzhen, Shenzhen, 518055, China. E-mail: zhangj9@sustech.edu.cn. Research of this author was supported by National Natural Science Foundation of China (12222106, 12326605), Guangdong Basic and Applied Basic Research Foundation (No. 2022B1515020082).}
}
\date{}
\maketitle

\abstract{This paper introduces and develops novel coderivative-based Newton methods with Wolfe linesearch conditions to solve various classes of problems in nonsmooth optimization. We first propose a generalized regularized Newton method with Wolfe linesearch (GRNM-W) for unconstrained $C^{1,1}$ minimization problems (which are second-order nonsmooth) and establish global as well as local superlinear convergence of their iterates. To deal with convex composite minimization problems (which are first-order nonsmooth and can be constrained), we combine the proposed GRNM-W with two algorithmic frameworks: the forward-backward envelope and the augmented Lagrangian method resulting in the two new algorithms called CNFB and CNAL, respectively. Finally, we present numerical results to solve Lasso and support vector machine problems appearing in, e.g., machine learning and statistics, which demonstrate the efficiency of the proposed algorithms.}

\vspace*{-0.15in}

\section{Introduction}\label{intro}\vspace*{-0.05in}

The classical Newton method is an efficient second-order algorithm for unconstrained \(C^2\)-smooth minimization problems, with superlinear or quadratic local convergence when the Hessian of the objective function is positive-definite around a solution. Nevertheless, it has some serious drawbacks such as lack of global convergence, high computational cost, and restricted applicability. Accordingly, many modifications of Newton's method have been proposed, including damped Newton method, regularized Newton method, quasi-Newton methods, trust-region Newton methods, cubic regularized Newton method, etc. For such Newton-type methods, second-order smoothness of the objective function is required either in the problem formulation, or in the convergence analysis. However, in many natural models, the objective function is not second-order differentiable while we still want to utilize some generalized second-order derivatives to design {\em nonsmooth} Newton algorithms with the hope that they retain the fast convergence of Newton's method. With that in mind, we first focus our attention on the class of unconstrained \(C^{1,1}\) minimization problems, where the objective functions are continuously differentiable with locally Lipschitz continuous gradients. In this way, several nonsmooth Newton methods employing different generalized differentiation constructions have been proposed and developed in the literature. The most popular by far generalized Newton method is known as the {\em semismooth Newton method} (SNM), which primarily addresses, along with its various modifications, to solving Lipschitzian gradient equations that arise, e.g., from stationary conditions for minimizing ${\cal C}^{1,1}$ functions as well as systems that can be reduced to this framework. The main analytic tools in SNM and its versions are {\em generalized Jacobians} by Clarke \cite{clarke1990optimization}. Among an enormous amount of publications on SNM and its modifications, we refer the reader to the books \cite{facchinei2003finite,Izmailov2014newton,kk-book} and the bibliographies therein for a variety of results, discussions, and historical comments. Due to the well-recognized limitations of SNM, discussed in the aforementioned publications, the search of other nonsmooth Newtonian methods has been undertaken over the years. The reader can find more information on such developments in the fresh book \cite{m24}.

This paper develops a quite recent direction in variational theory and applications of nonsmooth Newton methods with algorithms constructed by employing {\em coderivatives} by Mordukhovich \cite{mordukhovich2006variational} instead of generalized Jacobians by Clarke. Newtonian iterations to find stationary points and local minimizers are defined in this way by using coderivative-based {\em second-order subdifferentials/generalized Hessians} of objective functions in the sense of \cite{mordukhovich1992sensitivity,mordukhovich2006variational}. The latter constructions of second-order variational analysis enjoy comprehensive calculus rules and admit explicit calculations in terms of the given data for broad classes of nonsmooth functions that overwhelmingly appear in various settings of optimization and its applications to machine learning, data science, statistics, biochemical modeling, etc.; see \cite{m24} for more details and references. A variety of coderivative-based Newtonian algorithms have been developed and applied in \cite{aragon,duy2023generalized,khanh2022generalized,khanh2023globally,
m24,mordukhovich2021generalized} with establishing their local and global convergence, convergence rates, and applications to practical modeling. To ensure {\em global convergence} of the generalized damped Newton method in  \cite{duy2023generalized,khanh2022generalized} and the generalized regularized Newton method in \cite{khanh2023globally}, together with their algorithmic implementations for various classes of optimization problems, the backtracking {\em Armijo linesearch} has been widely employed.

In this paper, we first propose and justify a generalized regularized Newton method with the {\em Wolfe linesearch} (GRNM-W) for unconstrained $C^{1,1}$ minimization problems. Its iterative procedure is given in the following form:
\[x^{k+1}=x^k+\tau_k d^k\;\text{ with }\;-\nabla\vp(x^k)\in \partial^2\vp(x^k)(d^k)+\mu_k B_k d^k,\]
where \(\tau_k\) is a stepsize determined by a linesearch that satisfies the Wolfe conditions, where \(\partial^2\vp(x^k):=(D^*\nabla\vp)(x^k)\) is Mordukhovich's generalized Hessian (i.e., coderivative of the gradient mapping), where \(\mu_k>0\) is a regularization parameter, and where $B_k\succ 0$ is a regularization matrix.
GRNM-W has two {\em crucial differences} from the generalized regularized Newton method (GRNM) in \cite{khanh2023globally}. On one hand, GRNM-W uses a Wolfe linesearch strategy instead of the backtracking linesearch. On the other hand, GRNM-W incorporates a more general regularization matrix than GRNM (which uses the identity matrix).
We establish {\em global convergence} and {\em local superlinear convergence rates} of GRNM-W under the same assumptions as in \cite{khanh2023globally}. Moreover, we present an appropriately {\em modified version} of GRNM-W so that it can be applied to arbitrary nonconvex functions. Global convergence with convergence rates of the modified GRNM-W are established for general nonconvex functions satisfying the {\em Kurdyka-\L ojasiewicz} (KL) property. To the best of our knowledge, the obtained results are new for nonsmooth Newton-type methods. We show that the Wolfe linesearch is {\em more efficient} than the backtracking Armijo linesearch employed in \cite{khanh2023globally} and other nonsmooth Newton methods mentioned above, especially when evaluations of gradients are not too expensive and the starting point is far enough from a solution. One particularly impressive property of the Wolfe linesearch for Newton-type methods is that it allows us to choose {\em larger-than-unit stepsizes} in the initial stage of the iteration.

Starting with unconstrained problems of $C^{1,1}$ optimization, we then extend our GRNM-W method to significantly more general classes of {\em constrained} optimization problems, which may be even {\em first-order nonsmooth}. Of our primary interest here are problems of {\em convex composite minimization}, where one of the functions in summation is extended-real-valued and hence incorporates constraints. We combine the coderivative-based GRNM-W with the two algorithmic frameworks: {\em forward-backward envelope} from  \cite{patrinos2013proximal} and {\em augmented Lagrangian method} from \cite{rockafellar1976augmented}. These combinations lead us to the  two new algorithms called CNFB (coderivative-based Newton forward-backward method) and CNAL (coderivative-based Newton augmented Lagrangian method). The main difference of CNFB from GRNM in  \cite{khanh2023globally} is the usage of the Wolfe linesearch instead of the Armijo one, while CNAL is different from SSNAL (semismooth Newton augmented Lagrangian method) in \cite{li2018efficiently} by using coderivatives instead of generalized Jacobians and the Wolfe conditions instead of the Armijo linesearch. We present applications of our results to {\em support vector machines} and {\em Lasso problems} with numerical experiments that demonstrate the efficiency of our algorithms.

The remaining parts of the paper are organized as follows. In
Section~\ref{sec_pre}, we review preliminaries from variational analysis including generalized differentiation constructions and the semismooth and semismooth$^*$ properties. Section~\ref{sec_reg_Wolfe} proposes and develops a generalized regularized Newton method with the Wolfe
linesearch (GRNM-W) for $C^{1,1}$ functions whose generalized Hessians are positive-semidefinite. Global convergence results  with local superlinear convergence rates are derived in this section. A modified version of GRNM-W is given in Section~\ref{sec_KL} in such a way that it is applicable to arbitrary nonconvex $C^{1,1}$ functions without any positive-semidefiniteness assumptions whatsoever. We then verify global convergence and convergence rates for the modified GRNM-W assuming the KL property of the objective function. Our study of convex composite minimization problems begins in Section~\ref{sec_convex composite}, where we propose the CNFB algorithm by combining GRNM-W with the forward-backward envelope. In Section~\ref{sec_CNAL}, we propose the CNAL algorithm by embedding GRNM-W in the augmented Lagrangian method. Applications of the obtained results to support vector machines and Lasso problems together with the corresponding numerical experiments are reported in Section~\ref{sec_svm} and Section~\ref{sec_numerical}, respectively. Finally,  Section~\ref{sec_conclusions} presents concluding remarks and discusses some directions of our future research.

Our notations are standard. Recall that \(\mathbb{N}:=\{0,1,\ldots\}\), \(\mathbb{R}\) is the field of real numbers, and \(\Rb:=\mathbb{R}\cup\{\infty\}\). The distance between \(x\in\mathbb{R}^n\) and a nonempty set \(\Omega\subset\mathbb{R}^n\) is defined by \(\mathrm{dist}(x,\Omega):=\inf\{\|x-y\|\;|\;y\in\Omega\}\). The symbol \(\mathbb{B}_{\delta}(\bar{x})=\{x\in\mathbb{R}^n\;|\;\|x-\bar{x}\|<\delta\}\) stands for the open ball centered at \(\bar{x}\) with radius \(\delta\). For a symmetric matrix \(A\in\mathbb{R}^{n\times n}\), the notation \(A\succ 0\) means that \(A\) is positive-definite. The norm defined by such a matrix $A$ is denoted by $\|x\|_A:=\sqrt{\braket{x,Ax}}$.\vspace*{-0.15in}

\section{Preliminaries from Variational Analysis}\label{sec_pre}\vspace*{-0.05in}

Here we overview some well-known notions and results of variational analysis broadly employed in the paper; see the books \cite{mordukhovich2006variational,mordukhovich2018variational,m24,
rockafellar2009variational} for more details and related material.

For a nonempty set \(\Omega \subset \mathbb{R}^n\), the (Fr\'echet) {\em regular normal cone} to \(\Omega\) at \(\bar{x}\in\Omega\) is
\begin{equation}\label{rnc}
\widehat{N}_{\Omega}(\bar{x}):=\Big\{v\in\mathbb{R}^n\;\Big|\;\limsup_{x\overset{\Omega}{\to}\bar{x}}\frac{\braket{v,x-\bar{x}}}{\|x-\bar{x}\|}\le 0\Big\},
\end{equation}
where the symbol \(x\overset{\Omega}{\to}\bar{x}\) indicates that \(x\to\bar{x}\) with \(x\in\Omega\). The (Mordukhovich) {\em basic/limiting normal cone} to \(\Omega\) at \(\bar{x}\in\Omega\) is defined by
\begin{equation}\label{lnc}
N_\Omega(\bar x):=\big\{v\in\mathbb{R}^n\;\big|\;\exists\,x_k\overset{\Omega}{\to}\bar{x},v_k\to v\;\text{ as }\;k\to\infty\;\text{ with }\;v_k\in\widehat{N}_\Omega(x_k)\big\}.
\end{equation}
Note that the set $N_\Omega(\bar x)$ is often nonconvex as, e.g., for $\Omega:=\{(x,\alpha)\in\R^2\;|\;\alpha\ge -|x|\}$ at $\bar x=0$. Thus \eqref{lnc} cannot be generated in duality by any tangential approximation of $\Omega$ at $\bar x$ since duality always yields convexity. Nevertheless, the limiting normal cone \eqref{lnc} and the corresponding coderivative and subdifferential constructions enjoy {\em full calculus} based on {\em variational/extremal principles}.

Consider a set-valued mapping/multifunction \(F:\mathbb{R}^n \rightrightarrows \mathbb{R}^m\) with the graph $\gph F:=\{(x,y)\in\mathbb{R}^n\times\mathbb{R}^m\;|\;\,y\in F(x)\}$. Generated by the corresponding normal cone in \eqref{rnc} and \eqref{lnc}, the {\em regular coderivative} and {\em limiting coderivative} of $F$ at \((\bar{x},\bar{y})\in\gph F\) are defined, respectively, by
\begin{equation}\label{reg-cod}
\widehat{D}^*F(\bar{x},\bar{y})(v):=\big\{u \in \mathbb{R}^n\;\big|\;(u,-v)\in\widehat{N}_{\gph F}(\bar{x},\bar{y})\big\},\quad v\in\mathbb{R}^m,
\end{equation}
\begin{equation}\label{lim-cod}
D^*F(\bar{x},\bar{y})(v):=\big\{u \in \mathbb{R}^n\;\big|\;(u,-v)\in N_{\gph F}(\bar{x},\bar{y})\big\},\quad v\in\mathbb{R}^m.
\end{equation}
When $F\colon\R^n\to\R^m$ is single-valued, we omit {\(\bar{y}=F(\bar{x})\)} in the notations \eqref{reg-cod} and \eqref{lim-cod}. Recall further that a multifunction \(F:\mathbb{R}^n\rightrightarrows \mathbb{R}^m\) is {\em metrically regular} at \((\bar{x},\bar{y})\in\gph F\) if
there exist a positive constant \(\mu>0\) and neighborhoods \(U\) of \(\bar{x}\) and \(V\) of \(\bar{y}\) such that
\begin{equation}\label{mertreg}
\mathrm{dist}\left(x;F^{-1}(y)\right)\le\mu\,\mathrm{dist}\left(y;F(x)\right)\;\mbox{ for all }\;x,y)\in U\times V.
\end{equation}

Given now an extended-real-valued function \(\vp:\mathbb{R}^n\to\Rb\) with ${\rm dom}\,\vp:=\{x\in\R^n\;|\;\vp(x)<\infty\}$, define the {\em limiting subdifferential} of $\vp$ at \(\bar{x} \in \mathrm{dom}\,\vp\) geometrically by
\begin{equation}\label{sub}
\partial\vp(\bar{x}):=\big\{v\in\mathbb{R}^n\;\big|\;(v,-1)\in N_{\mathrm{epi}\,\vp}\left(\bar{x},\vp(\bar{x})\right)\big\}
\end{equation}
while observing that \eqref{sub} admits various analytic representations that can be found in the aforementioned books. The {\em second-order subdifferential} (or {\em generalized Hessian}) of $\vp$  at \(\bar{x}\in \mathrm{dom}\,\vp\) for \(\bar{v}\in\partial\vp(\bar{x})\) is defined in \cite{mordukhovich1992sensitivity} as the coderivative of the subgradient mapping \(\partial^2\vp(\bar{x},\bar{v}):\mathbb{R}^n\rightrightarrows\mathbb{R}^n\) by
\begin{equation}\label{2nd}
\partial^2\vp(\bar{x},\bar{v})(u):=\big(D^*\partial\vp\big)(\bar{x},\bar{v})(u)\;\mbox{ whenever }\;u\in\mathbb{R}^n.
\end{equation}
Note that if $\vp$ is a $C^2$-smooth function around $\bar x$, then \eqref{2nd} reduces to the classical Hessian matrix $\partial^2\vp(\bar{x})(u)=\{\nabla^2\vp(\bar x)(u)\}$ for all $u\in\R^n$.
Over the years, extensive calculus rules, explicit calculations, and a variety of applications have been obtained in terms of \eqref{2nd}, which have been summarized in the recent book \cite{m24} with the numerous references therein.

Among the striking applications of \eqref{2nd}, we mention {\em complete characterizations} in its terms several notions of {\em variational stability} in finite and infinite dimensions. It concerns, in particular, the following major stability notion in optimization introduced in \cite{poliquin1998tilt},

\begin{definition}[\bf tilt stability]\label{tilt} {\rm
For \(\vp:\mathbb{R}^n\to\Rb\), a point \(\bar{x}\in \mathrm{dom}\,\vp\) is called a {\em tilt-stable local minimizer} of \(\vp\) if there exists
a number \(\gamma>0\) such that the mapping
\[M_{\gamma}:v\mapsto \mathrm{argmin}\big\{\vp(x)-\braket{v,x}\;\big|\;\,x\in \mathbb{B}_{\gamma}(\bar{x})\big\}\]
is single-valued and Lipschitz continuous on some neighborhood of \(0\in\mathbb{R}^n\) with \(M_{\gamma}(0)=\{\bar{x}\}\).
A Lipschitz constant of \(M_{\gamma}\) around \(0\) is called a {\em modulus of tilt stability} of \(\vp\) at \(\bar{x}\).}
\end{definition}

The next notion of semismoothness was introduced in \cite{mifflin1977semismooth} for real-valued functions and then extended to vector-valued functions in \cite{kummer,qi1993nonsmooth} with applications to generalized Newton's methods for directionally differentiable Lipschitz functions.

\begin{definition}[\bf semismoothness]\label{semi-lip}
{\rm A mapping \(f:\mathbb{R}^n \to \mathbb{R}^m\) is said to be {\em semismooth} at \(\bar{x}\) if it is locally Lipschitz continuous around \(\bar{x}\) and  the limit
\begin{equation}
\lim_{\substack{A\in \mathrm{co}\,\overline{\nabla}f(\bar{x}+tu')\\u'\to u,t\downarrow 0}}Au'
\end{equation}
exists for all \(u\in\mathbb{R}^n\), where `co' stands for the convex hull of the set in question, and where
\[\overline{\nabla}f(x):=\big\{A\in\mathbb{R}^{m\times n}\;\big|\;\exists\, x_k \overset{\Omega_f}{\to}x\;\mbox{ with }\;\nabla f(x_k)\to A\big\}\]
with \(\Omega_f\) signifying the set of points at which \(f\) is differentiable.}
\end{definition}

Note that the set $\overline{\nabla}f(x)$ above is also written as $\partial_B f(x)$ and is called the {\em B-subdifferential} of $f$ at $x$. The convex hull $\mathrm{co}\,\overline{\nabla}f(x)$ is also written as $\partial_C f(x)$ and is called the (Clarke) {\em generalized Jacobian} of $f$ at $x$. Observe the relationship (with $A^T$ standing for the matrix transposition)
\begin{equation*}
{\rm co}\,D^*f(\bar x)(v)=\big\{A^T\;\big|\;A\in\partial_C f(\bar x)\big\},\quad v\in\R^m,
\end{equation*}
between the limiting coderivative \eqref{lim-cod} and the generalized Jacobian of $f$ at $\bar x$, which is valid for any mapping $f\colon\R^n\to\R^m$ locally Lipschitzian around $\bar x$. The following characterization of semismoothness is taken from  \cite[Theorem~2.3]{qi1993nonsmooth}.

\begin{proposition}[\bf characterization of semismoothness]\label{semi-char} Let $f\colon\R^n\to\R^m$ be locally Lipschitzian around $\bar x$. Then $f$ is  semismooth at \(\bar{x}\) if and only if it is directionally differentiable at \(\bar{x}\) in every direction and for any \(x\to\bar{x}\) and \(A\in\partial_C f(\bar x)\) we have the condition
\begin{equation}\label{SAP}
f(x)-f(\bar{x})-A(x-\bar{x})=o(\|x-\bar{x}\|).
\end{equation}
\end{proposition}

To proceed further, recall that the {\em directional normal cone} to a set \(\Omega\subset\mathbb{R}^s\) at \(\bar{z}\in\Omega\) in the direction \(d\in\mathbb{R}^s\) is introduced in \cite{ginchev2010directionally} by
\begin{equation*}
N_{\Omega}(\bar{z};d):=\big\{v\in\mathbb{R}^s\;\big|\;\exists\,t_k\downarrow 0,\;d_k\to d,\;v_k\to v \;\textrm{ with }\;
v_k\in\widehat{N}_{\Omega}(\bar{z}+t_kd_k)\big\}.
\end{equation*}
The {\em directional coderivative} of \(F:\mathbb{R}^n\rightrightarrows\mathbb{R}^m\) at \((\bar{x},\bar{y})\in\gph F\) in the direction \((u,v)\in\mathbb{R}^n\times\mathbb{R}^m\) is
\begin{equation*}
D^*F\left((\bar{x},\bar{y});(u,v)\right)(q):=\big\{p\in\mathbb{R}^n\;\big|\;(p,-q)\in N_{\gph F}\big((\bar{x},\bar{y});(u,v))\big)\big\},\quad q\in\R^m,
\end{equation*}
as defined in \cite{gfrerer2013directional}. Using the latter construction, the semismoothness was extended in \cite{gfrerer2021semismooth} to set-valued mappings $F\colon\R^n\rightrightarrows\R^m$ as follows: $F$ is said to be {\em semismooth$^*$} at \((\bar{x},\bar{y})\in\gph F\) if for all \((u,v)\in\mathbb{R}^n\times\mathbb{R}^m\) we have the condition
\[\braket{p,u}=\braket{q,v}\;\mbox{ whenever }\;(q,p)\in\gph D^*F\big((\bar{x},\bar{y});(u,v)\big).\]
We refer the reader to \cite{gfrerer2021semismooth} and also to \cite[Section~9.1.2]{m24} for various properties of semismooth$^*$ mappings. Note, in particular, that if $F=f\colon\R^n\to\R^m$ is single-valued and locally Lipschitzian around $\bar x$, then its semismooth$^*$ property at $\bar x$ is equivalent to condition  \eqref{SAP} without assuming the directional differentiability of $f$ at $\bar x$ as in Proposition~\ref{semi-char}.\vspace*{0.03in}

In our algorithms developed below, we achieve several convergence rates recalled in what follows.

\begin{definition}[\bf convergence rates]\label{rates}
{\rm Let \(\{x^k\}\subset\mathbb{R}^n\) be a sequence converging to \(\bar{x}\in\mathbb{R}^n\) as \(k\to \infty\) with \(x^k\neq\bar{x}\) for all $k\in\mathbb N$.

{\bf(i)} The convergence rate  of $\{x^k\}$ is {\em R-linear} if there are constants \(\mu\in(0,1)\) and \(c>0\) such that \(\|x^k-\bar{x}\|\le c\mu^k\) for all \(k\) sufficiently large.

{\bf(ii)} The convergence rate of $\{x^k\}$ is {\em Q-linear} if there exists a constant \(\mu\in (0,1)\) such that \(\|x^{k+1}-\bar{x}\|\le \mu\|x^k-\bar{x}\|\) for all \(k\) sufficiently large.

{\bf(iii)} The convergence rate of $\{x^k\}$ is {\em Q-superlinear} if \(\lim\limits_{k\to\infty}\frac{\|x^{k+1}-\bar{x}\|}{\|x^k-\bar{x}\|}=0\).}
\end{definition}

Finally in this section, we present the following useful technical lemma
taken from \cite[Lemma~4]{khanh2023globally}.

\begin{lemma}[\bf linear convergence of sequences]\label{lemma:linearconvergence}
Let \(\{\alpha_k\},\{\beta_k\},\{\gamma_k\}\) be sequences of positive real numbers. Assume that there exist constants \(c_1,c_2,c_3>0\) satisfying the estimates for large $k\in\mathbb N$:

{\bf(i)} \(\alpha_k-\alpha_{k+1}\ge c_1\beta^2_k.\)

{\bf(ii)} \(\beta_k\ge c_2\gamma_k.\)

{\bf(iii)} \(c_3\gamma_k^2\ge\alpha_k.\)
\\
Then \(\{\alpha_k\}\) converges to zero Q-linearly while $\{\beta_k\}$ and $\{\gamma_k\}$ converge to zero R-linearly.
\end{lemma}\vspace*{-0.15in}

\section{Regularized Newtonian Method with Wolfe Linesearch}\label{sec_reg_Wolfe}

Consider the following unconstrained optimization problem:
\begin{equation}\label{P}
\mbox{minimize }\;\vp(x)\;\mbox{ subject to }\;x\in\mathbb{R}^n,
\end{equation}
where \(\vp\) is of {\em class $C^{1,1}$}, i.e., it is continuously differentiable with the locally Lipschitzian gradient. We propose and justify the {\em generalized regularized Newton method} with the {\em Wolfe linesearch} (GRNM-W) for solving the second-order nonsmooth problem (\ref{P}). Here is the algorithm:

\begin{algorithm}[H]
\caption{Generalized regularized Newton method with the Wolfe
linesearch (GRNM-W)}\label{alg_regularizedNewton_Wolfe}
{\bf Input:} \(x^0 \in \mathbb{R}^n\), \(c>0\), \(0 < \sigma_1 < \sigma_2 <1\), \(\rho \in (0,1]\).
\begin{algorithmic}[1]
\For{\(k=0,1,\ldots\)}
\State If \(\nabla \varphi (x^k)=0\), stop; otherwise set \(\mu_k = c\|\nabla \varphi(x^k)\|^{\rho}\) and go to next step.
\State Choose $B_k \succ 0$. Find \(d^k \in \mathbb{R}^n\) such that \(-\nabla \varphi (x^k) \in \partial^2 \varphi (x^k){(d^k)} + \mu_k B_k d^k\).
\State Set \(\tau_k=1\) and check the Wolfe conditions:
\[\varphi(x^k+\tau_k d^k)\le\varphi(x^k)+\sigma_1 \tau_k \braket{\nabla\varphi(x^k),d^k},\]
\[\braket{\nabla \varphi(x^k+\tau_k d^k),d^k} \ge \sigma_2 \braket{\nabla \varphi(x^k),d^k}.\]
If these conditions do not hold, adjust \(\tau_k\) (using any specific implementation of the Wolfe linesearch) until it satisfies the Wolfe conditions. We assume that in the implementation there exists an upper bound \(\tau_{\mathrm{max}}\) on the maximum stepsize allowed.
\State Set \(x^{k+1}=x^k+\tau_kd^k\).
\EndFor
\end{algorithmic}
\end{algorithm}

The proposed algorithm is a counterpart of the globally convergent coderivative-based GRNM from \cite{khanh2023globally} with replacing the Armijo linesearch by the Wolfe one. The reader can consult \cite{doikov,roma} and the references therein for some other versions of globally convergent regularized Newton methods in the case of convex $C^2$-smooth objective functions using the Armijo linesearch for globalization.\vspace*{0.03in}

To proceed with the justification of Algorithm~\ref{alg_regularizedNewton_Wolfe}, we first present the following lemma.

\begin{lemma}[\bf existence of Newton-Wolfe directions]\label{lemma:newton step}
Let \(\varphi:\mathbb{R}^n \to \mathbb{R}\) be of class \(C^{1,1}\) around a point \(x \in \mathbb{R}^n\) such that \(\nabla\varphi(x)\neq 0\) and \(\partial^2\varphi(x)\) is positive-semidefinite, i.e.,
\begin{equation}
\braket{z,u}\ge 0\;\textit{ for all }\;u \in \mathbb{R}^n\;\textit{ and all }\;z \in \partial^2\vp(x)(u).
\end{equation}
Then for any positive-definite symmetric matrix $B\in\R^{n\times n}$,
there exists $d\ne 0$ such that
\begin{equation*}
 -\nabla\varphi(x)\in\partial^2\varphi(x)(d)+ B d.
\end{equation*}
\end{lemma}
\begin{proof} This can be distilled from the similar proof of \cite[Theorem~3(i)]{khanh2023globally}, and hence is omitted.
\end{proof}

Now we get the well-posedness of the proposed CRNM-W.

\begin{theorem}[\bf well-posedness of
Algorithm~{\rm\ref{alg_regularizedNewton_Wolfe}}]
\label{thm:wellposedness} Let \(\varphi:\mathbb{R}^n \to \mathbb{R}\) be of class \(C^{1,1}\) and be bounded from below, and let \(x^0\in\mathbb{R}^n\) be a starting point such that for all \(x\in \mathbb{R}^n\)  with \(\vp(x)\le \vp(x^0)\), the generalized Hessian
\(\partial^2\vp(x)\) is
positive-semidefinite.\footnote{This condition always holds when $\vp$ is convex; see \cite{chieu2011characterizing}.} If \(\nabla\vp(x^k)\neq 0\), then there exists \(d^k \neq 0\) such that the condition in line {\rm 3} of Algorithm~{\rm\ref{alg_regularizedNewton_Wolfe}} holds and the Wolfe linesearch is well-defined. Moreover, Algorithm~{\rm\ref{alg_regularizedNewton_Wolfe}} either stops after finitely many iterations at a stationary point, or generates a sequence \(\{x^k\}\) such that \(\{\vp(x^k)\}\) is decreasing.
\end{theorem}
\begin{proof}
The existence of \(d^k\) follows from Lemma \ref{lemma:newton step}. Since \(\partial^2\vp(x^k)\) is positive-semidefinite (by the sufficient decrease condition and the assumption on $x^0$), and since \(-\nabla\vp(x^k)-\mu_k B_k d^k\in \partial^2\vp(x^k)(d^k)\), we get \(\braket{-\nabla\vp(x^k)-\mu_k B_k d^k,d^k}\ge 0\) giving us
\begin{equation}\label{eq_gradient>=d}
\braket{-\nabla\vp(x^k), d^k}\ge \mu_k \|d^k\|^2_{B_k}.
\end{equation}
Using \(\mu^k>0\) and \(d^k\neq 0\) ensures that \(\braket{-\nabla\vp(x^k), d^k}>0\). This tells us by \cite[Lemma~3.1]{nocedal1999numerical} that \(d^k\) is a descent direction and that a stepsize satisfying the Wolfe conditions always exists when the objective function is continuously differentiable and bounded from below. By design, Algorithm~\ref{alg_regularizedNewton_Wolfe} either stops after finitely many iterations (in which case it lands at a stationary point), or generates a sequence of iterates \(\{x^k\}\) such that \(\vp(x^{k+1})<\vp(x^k)\) for all \(k \in \mathbb{N}\).
\end{proof}

The next theorem establishes stationarity of accumulation points of iterates in GRNM-W.

\begin{theorem}[\bf stationarity of accumulation points in
Algorithm~{\rm\ref{alg_regularizedNewton_Wolfe}}]
\label{thm:weakconvergence_wolfe} Under the assumptions of
Theorem~{\rm\ref{thm:wellposedness}}, suppose that the eigenvalues of the regularization matrices \(B_k\) in Algorithm~{\rm\ref{alg_regularizedNewton_Wolfe}} are contained in \([\mb^2,\Mb^2]\) with \(\Mb\ge\mb>0\), which implies that \(\mb\|x\|\le\|x\|_{B_k}\le \Mb\|x\|\) for all \(x\in\R^n\).  Then every accumulation point of \(\{x^k\}\) is stationary for problem \eqref{P}.
\end{theorem}
\begin{proof}
It follows from Theorem~\ref{thm:wellposedness} that
Algorithm~\ref{alg_regularizedNewton_Wolfe} either stops after finitely many iterations (in which case it must land at a stationary point), or generates a sequence of iterates \(\{x^k\}\) such that \(\vp(x^{k+1})<\vp(x^k)\) for all \(k \in \mathbb{N}\). This tells us that
\(\{x^k\}\subset \Omega:=\{x \in \mathbb{R}^n\;|\;\vp(x)\le \vp(x^0)\}.\) Suppose that the algorithm does not stop after finitely many steps. Then we have \(\nabla\vp(x^k)\neq 0\) for all \(k \in \mathbb{N}\). Recall that $\mu_k:=c\|\nabla\vp(x^k)\|^\rho$ with $c>0$ and $\rho\in(0,1]$. Define the modified directions
\begin{equation}\label{dtilde}
\tilde{d}^k:=\|\nabla\vp(x^k)\|^{\rho-1}d^k,\quad k\in\mathbb N,
\end{equation}
which is possible due to \(\|\nabla\vp(x^k)\|>0.\) Let us show that the sequence \(\{\tilde{d}^k\}\) is bounded. Indeed, the algorithm design (line~3 of Algorithm~\ref{alg_regularizedNewton_Wolfe}) provides
\begin{equation*}
-\nabla\vp(x^k)-\mu_k B_k d^k \in \partial^2\vp(x^k)(d^k)\;\mbox{ for all }\;k\in\mathbb B,
\end{equation*}
which ensures by the positive-semidefiniteness of \(\partial^2\vp(x^k)\) that
\begin{equation}\label{eq_psd}
{\braket{-\nabla\vp(x^k),d^k} \ge \mu_k\|d^k\|^2_{B_k} \ge \mu_k \mb^2 \|d^k\|^2.}
\end{equation}
By the Cauchy-Schwarz inequality with the choice of the regularization parameter \(\mu_k=c\|\nabla\vp(x^k)\|^{\rho}\) and the regularization matrices $B_k$, we have
\begin{equation*}
c\|\nabla\vp(x^k)\|^{\rho} \mb^2 \|d^k\|^2 \le c\|\nabla\vp(x^k)\|^{\rho} \|d^k\|^2_{B_k} \le \|\nabla\vp(x^k)\| \|d^k\|,
\end{equation*}
which brings us to the estimate
\begin{equation*}
\|\nabla\vp(x^k)\|^{\rho-1} \|d^k\| \le \frac{1}{\mb^2 c},\quad k\in\mathbb N.
\end{equation*}
This justifies that the sequence of \(\|\tilde{d}^k\|=\|\nabla\vp(x^k)\|^{\rho-1} \|d^k\|\) is bounded.

Let \(\{x^{k_j}\}\) be a subsequence of \(\{x^k\}\) such that \(x^{k_j} \to \bar{x}\) as \(j \to \infty\), i.e., \(\bar{x}\) is an accumulation point of \(\{x^{k}\}\). Since \(\vp(\bar{x})\) is clearly an accumulation point of the decreasing sequence \(\{\vp(x^k)\}\), we have \(\vp(x^k) \to \vp(\bar{x})\) as \(k \to \infty\). It follows from the Wolfe conditions
\begin{equation*}
\vp(x^{k+1})-\vp(x^k) \le {\sigma_1} \tau_k\braket{\nabla\vp(x^k),d^k} <0
\end{equation*}
that the passage to the limit leads us to
\begin{equation}\label{eq_tau}
\lim_{k \to \infty} \tau_k\braket{\nabla\vp(x^k),d^k}=0.
\end{equation}
The boundedness of \(\{\tilde{d}^{k}\}\) gives us a convergent subsequence, which can be taken as \(\{\tilde{d}^{k_j}\}\) without loss of generality. We get that \(\{d^{k_j}\}\) is convergent with limit \(\bar{d}:=\lim_{j \to \infty} d^{k_j}=\lim_{j \to \infty} \|\nabla\vp(x^{k_j})\|^{1-\rho}\tilde{d}^{k_j}\) by the continuity of the function \(\|\cdot\|^{1-\rho}\) on $\R$, which is equivalent to $\rho\le 1$. Next we claim that
\begin{equation}\label{eq_innerproduct}
\braket{\nabla\vp(\bar{x}),\bar{d}}=0.
\end{equation}
If \(\limsup_{j \to \infty}\tau_{k_j} > 0,\) then (\ref{eq_innerproduct}) follows from (\ref{eq_tau}). Hence we only need to consider the case where $\limsup_{j \to \infty}\tau_{k_j}=0$, which is the same as \(\lim_{j \to \infty}\tau_{k_j}=0.\)
Note that
\begin{equation*}
\langle\nabla\vp(x^{k_j}+\tau_{k_j}d^{k_j}),d^{k_j}\rangle\ge\sigma_2 \langle\nabla\vp(x^{k_j}),d^{k_j}\rangle\;\mbox{ for all }\;j\in\mathbb N.
\end{equation*}
Letting \(j \to \infty\) leads to \(\braket{\nabla\vp(\bar{x}),\bar{d}}\ge 0\) since
\(\sigma_2<1\). On the other hand, taking limits in \(\braket{\nabla\vp(x^{k_j}),d^{k_j}}<0\) yields \(\braket{\nabla\vp(\bar{x}),\bar{d}}\le 0\), which justifies \eqref{eq_innerproduct} in this case as well.
Combining further (\ref{eq_innerproduct}) with (\ref{eq_psd}), we have the equalities
\begin{equation}\label{eq_zero}
\mb^2c\|\nabla\vp(\bar{x})\|^{\rho}\|\bar{d}\|^2=\mb^2\lim_{j \to \infty} \mu_{k_j} \|d^{k_j}\|^2=0.
\end{equation}
If \(\|\bar{d}\|\neq 0\), then \(\|\nabla\vp(\bar{x})\|^{\rho}=0\) by (\ref{eq_zero}). Since \(\rho > 0\), we get \(\|\nabla\vp(\bar{x})\|=0\).
If \(\|\bar{d}\|=0\), we can still show that $\|\nabla\vp(\bar x)\|=0$.
Indeed, recall that \(\vp\) is of class \(C^{1,1}\) around \(\bar{x}\) and $-\nabla\vp(x^k)-\mu_k B_k d^k \in \partial^2\vp(x^k)(d^k)$. Then it follows from \cite[Theorem~1.44]{mordukhovich2006variational} that there is \(l>0\) with
\[\left\|\nabla\vp(x^{k_j})+\mu_{k_j} B_{k_j} d^{k_j}\right\|\le l\left\|d^{k_j}\right\|\]
for all large \(j\). Letting \(j \to \infty\) verifies \(\|\nabla\vp(\bar{x})\| \le l \|\bar{d}\|=0\) and completes the proof.
\end{proof}

\begin{theorem}[\bf convergence and convergence rates of GRNM-W]\label{thm:convergence_wolfe}
 In the setting of Theorem~{\rm\ref{thm:weakconvergence_wolfe}}, let \(\bar{x}\) be an accumulation point of \(\{x^k\}\) such that \(\nabla\vp\) is metrically regular around \(\bar{x}\). Then \(\bar{x}\) is a tilt-stable local minimizer of \(\vp\), and Algorithm~{\rm \ref{alg_regularizedNewton_Wolfe}} converges to \(\bar{x}\) with the convergence rates as follows:

{\bf(i)} The sequence \(\{\vp(x^k)\}\) converges Q-linearly to \(\vp(\bar{x})\).

{\bf(ii)} The sequences \(\{x^k\}\) and \(\{\nabla\vp(x^k)\}\) converge  R-linearly to \(\bar{x}\) and \(0\), respectively.

{\bf(iii)} The convergence rates of \(\{x^k\}\), \(\{\vp(x^k)\}\), and \(\{\nabla\vp(x^k)\}\) are Q-superlinear if \(\nabla\vp\) is semismooth$^*$ at \(\bar{x}\) and one of the following two groups of conditions holds:

{\bf(a)} \(\nabla \vp\) is directionally differentiable at \(\bar{x}\) and \(\sigma_1 \in (0,\frac{1}{2})\).

{\bf(b)} \(\sigma_1\in(0,\frac{1}{2l\kappa})\) and \(\sigma_2 \in (1-\frac{\kappa}{l},1)\), where \(\kappa>0\) and \(l>0\) are moduli of the metric regularity and Lipschitz continuity of \(\nabla\vp\) around \(\bar{x}\), respectively.
\end{theorem}
\begin{proof}
We split the proof into the seven claims as follows.\\[1ex]
{\bf Claim 1:} \textit{\(\bar{x}\) is a tilt-stable local minimizer of \(\vp\).} By Theorem \ref{thm:weakconvergence_wolfe}, \(\bar{x}\) is a stationary point of \(\vp\) being such that $\vp(\bar x)\le\vp(x^0)$. The  imposed positive-semidefiniteness of the generalized Hessian \(\partial^2\vp(\bar{x})\) and the metric regularity of \(\nabla\vp\) around \(\bar{x}\) with modulus \(\kappa\) imply by \cite[Theorem~4.13]{drusvyatskiy2014second} that \(\bar{x}\) is a tilt-stable local minimizer of \(\vp\) with the same modulus $\kappa$.\\[1ex]
{\bf Claim~2:} \textit{Let \(\{x^{k_j}\}\) be a subsequence of \(\{x^k\}\) with \(x^{k_j} \to \bar{x}\) as \(j \to \infty\). Then the subsequence of stepsizes \(\tau_{k_j}\) in Algorithm~{\rm \ref{alg_regularizedNewton_Wolfe}} is bounded from below by a positive number $\gamma$ and satisfies
\begin{equation*}
\vp(x^{k_j})-\vp(x^{k_j+1})\ge \frac{{\sigma_1}\gamma}{\kappa}\|d^{k_j}\|^2\;\mbox{ for all large }\;j\in\mathbb N.
\end{equation*}}

To verify this, suppose on the contrary that \(\{\tau_{k_j}\}\) is not bounded from below by a positive number, and so there exists a subsequence of \(\{\tau_{k_j}\}\) that converges to \(0\). We can assume without loss of generality that \(\tau_{k_j}\to 0\) as \(j \to \infty\).
By the second-order characterization of tilt-stable minimizers from \cite[Theorem~3.5]{mordukhovich2015second} and \cite[Proposition~4.6]{chieu2017second}, there exists \(\delta>0\) such that
\begin{equation}\label{eq:strongconvexitysecondorder}
\braket{z,w}\ge \frac{1}{\kappa}\|w\|^2\,\,\,{\rm for\,\,all\,\,}z \in \partial^2\vp(x)(w),\;x\in\mathbb{B}_{\delta}(\bar{x}),\;\mbox{ and }\;w\in\mathbb{R}^n.
\end{equation}
Since \(-\nabla\vp(x^{k_j})-\mu_{k_j} B_{k_j}d^{k_j}\in\partial^2\vp(x^{k_j})(d^{k_j})\), it follows that
\begin{equation}\label{eq:nablavp}\braket{-\nabla\vp(x^{k_j}),d^{k_j}} \ge \left(\mu_{k_j}\mb^2+\frac{1}{\kappa}\right)\|d^{k_j}\|^2\ge\frac{1}{\kappa}\|d^{k_j}\|^2
\,\,\,{\rm for\,\,large\,\,}j\in\mathbb N.
\end{equation}
Combining the Cauchy-Schwarz inequality, the Lipschitz continuity of \(\nabla\vp\) around $\bar x$ {with Lipschitz constant \(l\)}, and the Wolfe conditions yields
\[l\tau_{k_j}\|d^{k_j}\|^2\ge \braket{\nabla\vp(x_{k_j}+\tau_{k_j}d^{k_j})-\nabla\vp(x^{k_j}),d^{k_j}}\ge-(1-\sigma_2)
\braket{\nabla\vp(x^{k_j}),d^{k_j}}.\]
This tells us together with the estimates in (\ref{eq:nablavp}) that
{\[\tau_{k_j} \ge \frac{(1-\sigma_2)\braket{-\nabla\vp(x^{k_j}),d^{k_j}}}{l\|d^{k_j}\|^2}
\ge\frac{1-\sigma_2}{l\kappa}>0,\] }
which shows in turn that the subsequence \(\{\tau_{k_j}\}\) is bounded from below by $\gamma:=\frac{1-\sigma_2}{l\kappa}$.
Using finally the Wolfe conditions and (\ref{eq:nablavp}) justifies the  claimed assertion by
\begin{equation*}
\vp(x^{k_j})-\vp(x^{k_j+1})\ge{\sigma_1}\tau_{k_j}\braket{-\nabla\vp(x^{k_j}),d^{k_j}}\ge\frac{{\sigma_1}\gamma}{\kappa}\|d^{k_j}\|^2
\,\,\,{\rm for\,\,all\,\,large\,\,}j.
\end{equation*}
{\bf Claim~3:} \textit{The sequence \(\{x^k\}\) is convergent.} We show this by applying the convergence criterion based on Ostrowski's condition \cite[Proposition~8.3.10]{facchinei2003finite}. Let us first check that \(\bar{x}\) is an isolated accumulation point of \(\{x^k\}\). Indeed, if \(\tilde{x}\in \mathbb{B}_{\delta}(\bar{x})\) is an accumulation point of \(\{x^k\}\), we get by Theorem~\ref{thm:weakconvergence_wolfe} that \(\tilde{x}\) is a stationary point of \(\vp\). It follows from (\ref{eq:strongconvexitysecondorder}) and the second-order characterization of strong convexity for \(C^{1,1}\) functions in \cite[Theorem~5.2(i)]{chieu2011characterizing} that \(\vp\) is strongly convex with modulus $\kappa^{-1}$ on $\mathbb{B}_{\delta}(\bar{x})$, which ensures that  \(\tilde{x}=\bar{x}\). To verify further Ostrowski's condition, let \(\{x^{k_j}\}\) be a subsequence of \(\{x^k\}\) that converges to \(\bar{x}\). We
need to show that \(\lim_{j\to \infty}\|x^{k_j+1}-x^{k_j}\|=0\). To see this, deduce from  Claim~2 that
\begin{equation*}
\left\| x^{k_j+1}-x^{k_j} \right\|^2=\tau_{k_j}^2\|d^{k_j}\|^2\le {\tau_{\mathrm{max}}} \|d^{k_j}\|^2\le \frac{ {\tau_{\mathrm{max}}}\kappa}{{\sigma_1}\gamma}\left(\vp(x^{k_j})
-\vp(x^{k_j+1})\right)\to 0
\end{equation*}
as \(j \to \infty\).
Applying finally \cite[Proposition~8.3.10]{facchinei2003finite} yields the convergence of \(\{x^k\}\) to \(\bar{x}\) as \(k\to\infty\).\\[1ex]
{\bf Claim~4:} \textit{\(\{\vp(x^k)\}\) converges at least Q-linearly while \(\{x^k\}\) and \(\{\nabla\vp(x^k)\}\) converge at least R-linearly.} We use the strong convexity of \(\vp\) with modulus $\kappa^{-1}$ on \(\mathbb{B}_{\delta}(\bar{x})\) to get the estimates
\begin{equation}\label{eq:estimatestrongconvexity}\vp(x)\ge \vp(u)+\braket{\nabla\vp(u),x-u}+\frac{1}{2\kappa}\|x-u\|^2,
\end{equation}
\begin{equation}\label{eq:estimatestrongconvexity'}
\braket{\nabla\vp(x)-\nabla\vp(u),x-u}\ge\frac{1}{\kappa}\|x-u\|^2
\end{equation}
for all \(x,u \in \mathbb{B}_{\delta}(\bar{x})\). Since \(x^k \to \bar{x}\), it shows that \(x^k\in \mathbb{B}_{\delta}(\bar{x})\) for all $k$ sufficiently large.
Letting \(x=x^k, u=\bar{x}\) and using the Cauchy-Schwarz inequality
lead us to the conditions
\begin{equation}\label{eq:estimatefunctionvalue}
\vp(x^k)\ge \vp(\bar{x})+\frac{1}{2\kappa}\|x^k-\bar{x}\|^2
\end{equation}
\begin{equation}\label{eq:estimatenormnabla}\|\nabla\vp(x^k)\|\ge \frac{1}{\kappa}\|x^k-\bar{x}\|.
\end{equation}
By \cite[Lemma~A.11]{Izmailov2014newton}, the  Lipschitz continuity of \(\nabla\vp\) around \(\bar{x}\) with modulus \(l>0\) implies that
\begin{equation}\label{eq:estimatelip}
\vp(x^k)-\vp(\bar{x})=\left\vert\vp(x^k)-\vp(\bar{x})-\braket{\nabla\vp(\bar{x}),x^k-\bar{x}}\right\vert
\le \frac{l}{2}\|x^k-\bar{x}\|^2.
\end{equation}
By \(-\nabla\vp(x^k)-\mu_k B_k d^k\in\partial^2\vp(x^k)(d^k)\), it follows from \cite[Theorem 1.44]{mordukhovich2006variational} that
\begin{equation}
\left\|\nabla\vp(x^k)+\mu_k B_k d^k\right\| \le l\|d^k\|.
\end{equation}
Since \(x^k\to \bar{x}\) and \(\nabla\vp(\bar{x})=0,\) we have \(\mu_k=c\|\nabla\vp(x^k)\|^{\rho}\to 0\) as \(k\to \infty\), and so \(\mu_k\le l\) when \(k\) is large.
Thus \(\|\nabla\vp(x^k)\| \le \|\nabla\vp(x^k)+\mu_kd^k\|+\mu_k\|d^k\| \le 2l\|d^k\|.\)
Together with Claim~2, this yields
\begin{equation}\label{eq:estimatecompare}\vp(x^k)-\vp(x^{k+1}) \ge \frac{{\sigma_1}\gamma}{\kappa}\|d^k\|^2
\ge \frac{{\sigma_1}\gamma}{4\kappa l^2}\|\nabla\vp(x^k)\|^2.\end{equation}
Combining the estimates in (\ref{eq:estimatenormnabla})--(\ref{eq:estimatecompare}) and applying
Lemma~\ref{lemma:linearconvergence} with the choices of \(\alpha_k=\vp(x^k)-\vp(\bar{x})\),\;\(\beta_k=\|\nabla\vp(x^k)\|\),\;\(\gamma_k=\|x^k-\bar{x}\|\),\;\(c_1=\frac{\sigma\gamma}{4\kappa l^2}\), \(c_2=\frac{1}{\kappa},\) and \(c_3=\frac{l}{2}\) justifies this claim.\\[1ex]
{\bf Claim~5:} \textit{If \(\nabla\vp\) is semismooth\({}^*\) at \(\bar{x}\), then
\(\|x^k+d^k-\bar{x}\|=o(\|x^k-\bar{x}\|)\)
as \(k \to \infty\).} To verify this, deduce from the subadditivity property of coderivatives in \cite[Lemma~5.6]{khanh2022generalized} that
\begin{equation*}
\partial^2\vp(x^k)(d^k) \subset \partial^2 \vp (x^k)(x^k+d^k-\bar{x})+\partial^2\vp(x^k)(-x^k+\bar{x}).
\end{equation*}
On the other hand, by \(-\nabla\vp(x^k)-\mu_k B_k d^k \in \partial^2\vp(x^k)(d^k)\), there is \(v^k\in \partial^2\vp(x^k)(-x^k+\bar{x})\) with
\[-\nabla\vp(x^k)-\mu_k B_k d^k-v^k \in \partial^2\vp(x^k)(x^k+d^k-\bar{x}).\]
It follows from(\ref{eq:strongconvexitysecondorder}) and the
Cauchy-Schwarz inequality that
\begin{equation}\label{eq:estimate1}
\|x^k+d^k-\bar{x}\|
\le \kappa\left\|\nabla\vp(x^k)+v^k+\mu_k B_k d^k\right\|
\le \kappa\left(\|\nabla\vp(x^k)-\nabla\vp(\bar{x})+v^k\|+\mu_k\|B_kd^k\|\right).
\end{equation}
Combining \(-\nabla\vp(x^k)-\mu_k B_k d^k\in\partial^2\vp(x^k)(d^k)\) and (\ref{eq:strongconvexitysecondorder}) gives us
\begin{equation}\label{eq:estimateangle}
\braket{-\nabla\vp(x^k),d^k}\ge(\kappa^{-1}+\mu_k\mb^2)\|d^k\|^2\ge \kappa^{-1}\|d^k\|^2.
\end{equation}
Using the Cauchy-Schwarz inequality again together with  the Lipschitz continuity of \(\nabla\vp\), we have
\begin{equation}\label{eq:estimate2}
\|d^k\|\le\kappa \|\nabla\vp(x^k)\|=\kappa\|\nabla\vp(x^k)-\nabla\vp(\bar{x})\|
\le \kappa l \|x^k-\bar{x}\|.
\end{equation}
Moreover, the Lipschitz continuity of \(\nabla\vp\) on \(\mathbb{B}_{\delta}(\bar{x})\) guarantees that
\begin{equation}\label{eq:estimate3}
\mu_k=c\|\nabla\vp(x^k)\|^{\rho}=c\|\nabla\vp(x^k)-\nabla\vp(\bar{x})\|^{\rho}\le c\, l^\rho\|x^k-\bar{x}\|^{\rho}.
\end{equation}
Employing now the semismooth$^*$ property of of \(\nabla\vp\) at \(\bar{x}\) and the inclusion \(v^k\in \partial^2\vp(x^k)(-x^k+\bar{x})\)
allows us to deduce from \cite[Lemma~5.2]{duy2023generalized} that
\begin{equation}\label{eq:estimate4}
\|\nabla\vp(x^k)+v^k\|=\|\nabla\vp(x^k)-\nabla\vp(\bar{x})+v^k\|=o(\|x^k-\bar{x}\|).
\end{equation}
Combining finally the estimates in (\ref{eq:estimate1}) and (\ref{eq:estimate2})--(\ref{eq:estimate4}) together with \(\rho \in (0,1]\) tells us that
\begin{equation}\label{eq:43}
\begin{aligned}\|x^k+d^k-\bar{x}\|
\le& \kappa\left(\|\nabla\vp(x^k)-\nabla\vp(\bar{x})+v^k\|+\mu_k\|d^k\|\right)\\
\le&\kappa\, o(\|x^k-\bar{x}\|)+\kappa c \,l^\rho \|x^k-\bar{x}\|^{\rho}\kappa l \|x^k-\bar{x}\|\\
=&o(\|x^k-\bar{x}\|)+O(\|x^k-\bar{x}\|^{1+\rho})\\
=&o(\|x^k-\bar{x}\|)\;\mbox{ as }\;k\to\infty,
\end{aligned}
\end{equation}
which justifies the claimed convergence rate.\\[1ex]
{\bf Claim~6:} \textit{{We have \(\tau_k=1\) for all \(k\) sufficiently large if \(\nabla\vp\) is semismooth$^*$ at \(\bar{x}\), and if either condition {\rm (a)} or condition {\rm (b)} of this theorem holds.}} Observe that in case (a), the directional differentiability and semismoothness$^*$ of $\nabla \vp$ at $\bar{x}$ ensure by \cite[Corollary~3.8]{gfrerer2021semismooth} that the gradient mapping $\nabla \vp$ is semismooth at $\bar{x}$. From (\ref{eq:estimateangle}), we conclude by \cite[Proposition~8.3.18,(c)\(\implies\)(d)]{facchinei2003finite} that the Newton direction satisfies the sufficient decrease condition. Let us check that it also satisfies the Wolfe curvature condition therein. To furnish this, denote \(\partial^2_C\vp:=\mathrm{co\,}\overline{\nabla}(\nabla \vp)\) and deduce from the semismoothness of \(\nabla\vp\) at \(\bar{x}\) that
\begin{equation*}
\nabla\vp(x^k+d^k)=\nabla\vp(\bar{x})+H_k(x^k+d^k-\bar{x})+o(\|x^k+d^k-\bar{x}\|),
\end{equation*}
where \(H^k \in \partial^2_C\vp(x^k+d^k)\). By (\ref{eq:43}), we have \(\|x^k+d^k-\bar{x}\|=o(\|x^k-\bar{x}\|)\). It follows from \cite[Proposition~2.6.2(d)]{clarke1990optimization} that the matrix sequence \(\{H_k\}\) is bounded, and thus \(\nabla\vp(x^k+d^k)=o(\|x^k-\bar{x}\|)\). Employing now \cite[Lemma~7.5.7]{facchinei2003finite} yields \(\lim\limits_{k\to\infty}\frac{\|x^k-\bar{x}\|}{\|d^k\|}=1\), and hence
\begin{equation}\label{eq:nabla_is_small_o}
\nabla\vp(x^k+d^k)=o(\|d^k\|).\end{equation}
It follows therefore that for large $k\in\mathbb N$, we get the relationships
\begin{equation*}
\begin{aligned}
\braket{\nabla\vp(x^k+d^k),d^k}-\sigma_2\braket{\nabla\vp(x^k),d^k}
&=-\sigma_2\braket{\nabla\vp(x^k),d^k}+o(\|d^k\|^2)\\
&\ge\frac{\sigma_2}{\kappa}\|d^k\|^2+o(\|d^k\|^2)>0,
\end{aligned}
\end{equation*}
where the equality holds by (\ref{eq:nabla_is_small_o}) and the inequality is valid due to (\ref{eq:estimateangle}).

In case (b), we deduce from estimate (\ref{eq:estimatestrongconvexity}) that the Newton direction satisfies the sufficient decrease condition by \cite[Lemma~1]{khanh2023globally}. To check that the Wolfe curvature condition also holds, observe by \(\sigma_2>1-\frac{1}{l\kappa}\) and \(\lim\limits_{k\to\infty}\frac{\|x^k-\bar{x}\|}{\|d^k\|}=1\) (due to \cite[Lemma~7.5.7]{facchinei2003finite}) that
\begin{equation*}
\begin{aligned}
&\braket{\nabla\vp(x^k+d^k),-d^k}-\sigma_2\braket{\nabla\vp(x^k),-d^k}\\
\le& \braket{\nabla\vp(x^k),-d^k}-\frac{1}{\kappa}\|d^k\|^2 -\sigma_2\braket{\nabla\vp(x^k),-d^k}\\
\le& \|d^k\|^2\Big(-\frac{1}{\kappa}+(1-\sigma_2)l\frac{\|x^k-\bar{x}\|}{\|d^k\|}\Big)<0,
\end{aligned}
\end{equation*}
where the first inequality follows from (\ref{eq:estimatestrongconvexity'}). Thus the claim is justified.\\[1ex]
{\bf Claim~7:} \textit{The assertions about superlinear convergence hold.} To verify this, we have by Claim~6 that \(\tau_k=1\) for large \(k\), and hence
\begin{equation*}
\|x^{k+1}-\bar{x}\|=\|x^k+\tau_k d^k-\bar{x}\|=\|x^k+d^k-\bar{x}\|=o(\|x^k-\bar{x}\|)\;\mbox{ as }\;k\to\infty.
\end{equation*}
Then the Q-superlinear convergence of \(\{\vp(x^k)\}\) follows from estimates (\ref{eq:estimatefunctionvalue}) and (\ref{eq:estimatelip}), while the Q-superlinear convergence of \(\{\nabla\vp(x^k)\}\) follows from estimate (\ref{eq:estimatenormnabla}) and the Lipschitz continuity of \(\nabla\vp\). Combining Claims~1--7, we thus validate all the assertions of the theorem.
\end{proof}
\vspace*{-0.2in}

\section{Modified GRNM-W under KL Conditions}\label{sec_KL}
\vspace*{-0.05in}

A characteristic feature of GRNM-W (the same as for its ``Armijo" predecessor in \cite{khanh2023globally}) is the {\em positive-semidefiniteness} assumption on the generalized Hessian required for its well-posedness (i.e., the existence of iterates satisfying the algorithmic procedures); see
Theorem~\ref{thm:wellposedness}. In this section, we propose a {\em modification} of GRNM-W, labeled as {\em GRNM-WM}, which is well-posed for general nonconvex functions of class $C^{1,1}$ without {\em any requirements} on the generalized Hessian. Then we establish global convergence with explicit convergence rates for GRNM-WM under the fulfillment of the corresponding {\em Kurdyka-\L ojasiewicz} (KL) conditions.\vspace*{0.05in}

Here is the proposed algorithm for {\em general $C^{1,1}$ functions}.
\newpage
\begin{algorithm}
\caption{Modified regularized Newton method with Wolfe linesearch (GRNM-WM)}\label{alg_wolfe_KL}
{\bf Input:} \(x^0 \in \mathbb{R}^n\), \(c>0\), \(0 < \sigma_1 < \sigma_2 <1\), \(M>m>0\).
\begin{algorithmic}[1]
\For{\(k=0,1,\ldots\)}
\State If \(\nabla \varphi (x^k)=0\), stop; otherwise go to next step.
\State Choose \(\mu_k\ge 0\) in such a way that there exists \(d^k \in \mathbb{R}^n\) with \[-\nabla \varphi (x^k) \in \partial^2 \varphi (x^k){(d^k)} + \mu_k d^k,\]\[m\|d^k\|^2\le\braket{-\nabla\vp(x^k),d^k},\quad \|\nabla\vp(x^k)\|\le M\|d^k\|.\]
\State Set \(\tau_k=1\) and check the Wolfe conditions:
\[\varphi(x^k+\tau_k d^k)\le\varphi(x^k)+\sigma_1 \tau_k \braket{\nabla\varphi(x^k),d^k},\]
\[\braket{\nabla \varphi(x^k+\tau_k d^k),d^k} \ge \sigma_2 \braket{\nabla \varphi(x^k),d^k}.\]
If the conditions do not hold, adjust \(\tau_k\) until it satisfies the Wolfe conditions. We assume that in the implementation there exista an upper bound \(\tau_{\mathrm{max}}\) on the allowed maximum stepsize.
\State Set \(x^{k+1}=x^k+\tau_kd^k\).
\EndFor
\end{algorithmic}
\end{algorithm}

We first show that Algorithm~\ref{alg_wolfe_KL} is {\em well-posed} in the general nonconvex setting.

\begin{theorem}[\bf well-posedness of Algorithm~\ref{alg_wolfe_KL}] \label{lemma:wellposedness_of_alg2}
Suppose that the set \(\Omega=\{x\in \mathbb{R}^n\;|\;\vp(x)\le \vp(x^0)\}\) is bounded. Then for any $k\in\mathbb N$, there exists \(\mu_k\) such that the inclusion in line 3 of {\rm Algorithm \ref{alg_wolfe_KL}} is solvable and the required conditions are satisfied.
\end{theorem}
\begin{proof}
Recall that \(\emptyset\neq\partial_B(\nabla\vp)(x^k)d\subset \partial^2\vp(x^k)(d)\) by \cite[Proposition~1.51]{Izmailov2014newton} and \cite[Remark~5(ii)]{khanh2023globally}. Pick any matrix \(H_k\in\partial_B(\nabla\vp)(x^k)\), which is always symmetric. Then the linear system \((H_k+\mu_k)d=-\nabla\vp(x^k)\) is solvable if \(\mu_k>\max(0,-\lambda_{\rm min}(H_k))\), where \(\lambda_{\rm min}(H_k)\) is the smallest eigenvalue of \(H_k\) that may be negative since \(\vp\) is nonconvex. Moreover, the solution is nonzero whenever $\nabla\vp(x^k)$ is nonzero. Choose \(\mu_k=\max(0,-\lambda_{\rm min}(H_k))+m\) with \(m>0\) and observe that the symmetric matrix \(H_k+\mu_k I\) is positive-definite and that \(\braket{-\nabla\vp(x^k),d^k}=\braket{(H_k+\mu_k I)d^k,d^k}\ge m\|d^k\|^2\). Furthermore, since \(\nabla\vp\) is locally Lipschitzian on the compact set \(\Omega\), it is Lipschitz continuous on \(\Omega\) with some constant \(l>0\). By \cite[Theorem~1.44]{mordukhovich2006variational}, we have that $\|w\|\le l\|d\|$ for any $d$ and any $w\in\partial^2\vp(x^k)(d)$. Combining the latter with the definition of $\mu_k$ brings us to
\begin{equation*}
\|\nabla\vp(x^k)\|=\|(H_k+\mu_k I)d^k\|\le \|H_k d^k\|+\|\mu_k d^k\|\le l\|d^k\|+l\|d^k\|+m\|d^k\|=M\|d^k\|
\end{equation*}
with \(M:=2l+m>m\) and thus completes the proof of the theorem.
\end{proof}

\begin{remark}[\bf discussions on CRNM-WM]\label{discAlg2} $\,$

{\rm{\bf(i)} It follows from the proof of
Theorem~\ref{lemma:wellposedness_of_alg2} that the regularization parameter \(\mu_k\) can be chosen as \(\mu_k:=l+m\), where \(m>0\) is arbitrary and \(l\) is the Lipschitz constant of \(\nabla\vp\) on \(\Omega\). Observe that, even without the prior knowledge of \(l\), we can adaptively choose \(\mu_k\) in the following way: pick an arbitrary number \(m>0\) and set \(\mu_k:=m+\mu r^j\), where \(\mu>0\), \(r>1\) and \(j\ge 0\) is the first nonnegative integer such that the linear system is solvable and the conditions on \(d^k\) hold. This procedure terminates in {\em finitely many steps}.

{\bf(ii)} In Algorithm~\ref{alg_wolfe_KL}, we can also use a more {\em general regularization matrix} $B_k\succ 0$ similarly to
Algorithm~\ref{alg_regularizedNewton_Wolfe}. An appropriate modification of the proof of  Theorem~\ref{lemma:wellposedness_of_alg2} shows that such an algorithm is still well-posed. Although our convergence analysis below can be easily extended to this more general case, for simplicity we focus on the case of the identity regularization matrix.}
\end{remark}

To proceed further, we recall the versions of the KL property used in what follows; cf.\ \cite{attouch09}.

\begin{definition}[\bf KL conditions]\label{KL} {\rm Let \(f:\mathbb{R}^n\to\Rb\) be an extended-real-valued lower semicontinuous (l.s.c.) function, and let $\bar x\in{\rm dom}\,f$. We say that:

{\bf(i)} The {\em basic KL property} holds for $f$ at $\bar x$ if there exist a number \(\eta\in(0,\infty])\), a neighborhood \(U\) of \(\bar{x}\), and a function \(\psi:[0,\eta)\to\mathbb{R}_+\) such that \(\psi\) is concave and $C^1$-smooth on $(0,\eta)$ with \(\psi(0)=0\) and \(\psi'(s)>0\) for all \(s\in(0,\eta)\), and that we have
\begin{equation*}
\psi'\big(f(x)-f(\bar{x})\big)\, \mathrm{dist}\big(0,\partial f(x)\big) \ge 1\;\mbox{ for all }\;x\in{U}\cap\big\{x\in\R^n\;\big|\;f(\bar{x})<f(x)<f(\bar{x})+\eta\big\}.
\end{equation*}

{\bf(ii)} If \(\psi\) can be chosen in (i) as \(\psi(s)=cs^{1-\theta}\) with \(\theta\in[0,1)\) for some \(c>0\), then \(f\) satisfies the {\em exponent KL property} at $\bar x$ with the exponent \(\theta\).

{\bf(iii)} $f$ is a {\em KL function} $(${\em of exponent} $\theta\in[0,1))$ if $f$ enjoys the basic KL property $($of exponent $\theta$, respectively$)$ at every point $\bar x\in{\rm dom}\,f$.}
\end{definition}

It has been well recognized in optimization theory and applications that the KL conditions from Definition~\ref{KL} are satisfied for broad classes of functions and play a significant role in establishing global convergence results with convergence rates in various numerical algorithms; see, e.g., \cite{aragon,attouch09,attouch2010proximal,bento,khanh2023globally,
qian2023superlinear} among other publications. A large class of descent algorithms for which the KL conditions are instrumental for deriving impressive convergence properties is described in \cite{attouch2010proximal} via the following generic properties of iterative sequences:

{\bf(H1)} There exists \(a>0\) such that for all \(k\ge 0\) we have
\begin{equation}\label{H1}\tag{H1}
\vp(x^{k+1})\le \vp(x^k)-a\|x^{k+1}-x^k\|^2.
\end{equation}

{\bf(H2)} There exists \(b>0\) such that for all \(k\ge 0\) we have
\begin{equation}\label{H2}\tag{H2}
\|\nabla\vp(x^{k+1})\|\le b\|x^{k+1}-x^k\|.
\end{equation}

The next lemma presents convergence results under the KL properties for abstract descent algorithms satisfying conditions (H1) and (H2).

\begin{lemma}[\bf convergence of abstract descent algorithms under KL conditions]
\label{thm:qian-pan}
Let \(\vp:\mathbb{R}^n\to\mathbb{R}\) be of class \(C^{1,1}\), and let \(\{x^k\}\) satisfy conditions
${\rm(\ref{H1})}$ and ${\rm(\ref{H2})}$. The following hold:

{\bf(i)} If \(\vp\) is a KL function, then we have \(\sum_{k=0}^{\infty}\|x^{k+1}-x^k\|<\infty\). In particular, \(\{x^k\}\) converges to a stationary point \(\bar x\) as $k\to\infty$.

{\bf(ii)} If \(\vp\) is a KL function of exponent \(\theta\in(0,1)\), then \(\{x^k\}\) converges to \(\bar x\) with the rates:

\begin{itemize}

\item[\bf(a)] When \(\theta\in(0,\frac{1}{2})\), for any \(\epsilon>0\) and any large $k\in\mathbb N$ we have
\[\|x^{k+1}-\bar x\|\le\epsilon\|x^k-\bar x\|^{\frac{1}{2\theta}}.\]
If furthermore  $\bar x$ is a local minimizer of $\vp$, then there is no function $\vp$ of class $C^{1,1}$ satisfying the KL property at $\bar x$ with such an exponent.

\item[\bf(b)] When \(\theta=\frac{1}{2}\), there exist \(\gamma>0\) and \(q\in(0,1)\) such that for all large $k\in\mathbb N$ we have
\[\|x^k-\bar x\|\le\sum_{j=k}^{\infty}\|x^{j+1}-x^j\|\le\gamma q^k.\]

\item[\bf(c)] When \(\theta\in(\frac{1}{2},1)\), there exists \(\gamma>0\) such that
\[\|x^k-\bar x\|\le\sum_{j=k}^{\infty}\|x^{j+1}-x^j\|\le\gamma k^{\frac{1-\theta}{1-2\theta}}.\]
\end{itemize}
\end{lemma}
\begin{proof} Assertions (i) and (ii) in cases (b) and (c), as well as the convergence rates in case (a) of (ii) for arbitrary stationary points of $\vp$, are taken from \cite[Theorems~3.1 and 3.2]{qian2023superlinear} with a small rewording. The inconsistency between the class $C^{1,1}$ and the KL property of $\vp$ with exponent \(\theta\in(0,\frac{1}{2})\) at local minimizers of $\vp$ has been recently observed in \cite[Theorem~4]{bento}.
\end{proof}

\begin{remark}[\bf modified generic conditions]\label{rem-modified} {\em For some important descent algorithms, condition (H2) does not hold, while its replacement

{\bf(H2')} There exists \(b>0\) such that for all \(k\ge 0\) we have
\begin{equation*}
\|\nabla\vp(x^{k})\|\le b\|x^{k+1}-x^k\|
\end{equation*}
does; see more discussions in \cite{bento} on such algorithms for general l.s.c.\ functions. It follows from \cite[Theorem~2]{bento} that for any abstract algorithm satisfying (H1) and (H2'), the KL property of exponent \(\theta\in(0,\frac{1}{2})\) at a stationary point $\bar x$ of $\vp$ yields the {\em finite termination} of the algorithm.}
\end{remark}

The next lemma shows that the iterative sequence \(\{x^k\}\) generated by GRNM-WM (Algorithm~\ref{alg_wolfe_KL}) for minimizing of $C^{1,1}$ functions satisfies conditions (H1) and (H2).

\begin{lemma}[\bf embedding GRNM-WM into the generic scheme]\label{thm:H1-H2}
Let \(\vp:\mathbb{R}^n\to\mathbb{R}\) be a function of class \(C^{1,1}\), and let $x^0\in\R^n$ be a starting point. Suppose that the set \(\Omega=\{x\in \mathbb{R}^n\;|\;\vp(x)\le \vp(x^0)\}\) is bounded. Then the sequence \(\{x^k\}\) generated by Algorithm~{\rm\ref{alg_wolfe_KL}} satisfies conditions {\rm(\ref{H1})} and {\rm(\ref{H2})}.
\end{lemma}
\begin{proof}
To verify condition (\ref{H1}), we deduce from the algorithmic design that
\begin{equation*}
\begin{aligned}
\vp(x^{k+1})-\vp(x^k)&\le\sigma_1\tau_k\braket{\nabla\vp(x^k),d^k}\\
&=-\sigma_1\tau_k\braket{-\nabla\vp(x^k),d^k}\\
&\le-\sigma_1\tau_km\|d^k\|^2\\
&=-\frac{\sigma_1 m}{\tau_k}\|\tau_k d^k\|^2\\
&\le-\frac{\sigma_1 m}{\tau_{\rm max}}\|\tau_k d^k\|^2\\
&=-\frac{\sigma_1 m}{\tau_{\rm max}}\|x^{k+1}-x^k\|^2,
\end{aligned}
\end{equation*}
where the first inequality is the sufficient decrease property, the second one follows from the condition \(\braket{-\nabla\vp(x^k),d^k}\ge m\|d^k\|^2\) guaranteed in Algorithm~\ref{alg_wolfe_KL}, and the third inequality holds since $\tau_k$ is upper bounded by $\tau_{\rm max}$, the maximum stepsize allowed.

Next we show that condition (\ref{H2}) holds. Note that $\vp$ is Lipschitz continuous on $\Omega$ with some constant $l$ since $\Omega$ is compact and $\vp$ is locally Lipschitz continuous by the imposed assumptions. It follows from the curvature condition as in the proof of Claim~6 of Theorem~\ref{thm:convergence_wolfe} that
\begin{equation*}
(\sigma_2-1)\braket{\nabla\vp(x^k),d^k}\le \braket{\nabla\vp(x^{k+1})-\nabla\vp(x^k),d^k}\le l\tau_k\|d^k\|^2.
\end{equation*}
From this and \(m\|d^k\|^2\le\braket{-\nabla\vp(x^k),d^k}\), we deduce that \(\tau_k\ge\frac{(1-\sigma_2)m}{l}\).
Therefore,
\begin{equation*}
\begin{aligned}
\|\nabla\vp(x^{k+1})\|&\le\|\nabla\vp(x^{k+1})-\nabla\vp(x^k)\|+\|\nabla\vp(x^k)\|\\
&\le l\|x^{k+1}-x^k\|+M\|d^k\|\\
&=l\|x^{k+1}-x^k\|+\frac{M}{\tau_k}\|\tau_kd^k\|\\
&\le l\|x^{k+1}-x^k\|+\frac{Ml}{(1-\sigma_2)m}\|x^{k+1}-x^k\|\\
&=\left(l+\frac{Ml}{(1-\sigma_2)m}\right)\|x^{k+1}-x^k\|,
\end{aligned}
\end{equation*}
where the second inequality follows from the condition $\|\nabla\vp(x^k)\|\le M \|d^k\|$ in Algorithm \ref{alg_wolfe_KL}.
\end{proof}

The following theorem summarizes the main results for GRNM-WM obtained above.

\begin{theorem}[\bf performance of CRNM-WM]\label{converKL}
Let \(\vp:\mathbb{R}^n\to\mathbb{R}\) be of class \(C^{1,1}\), and let \(\{x^k\}\) be the sequence generated by Algorithm~{\rm\ref{alg_wolfe_KL}}. Suppose that the set \(\Omega=\{x\in\mathbb R^n\;|\; \vp(x)\le\vp(x^0)\}\) is bounded. Then Algorithm~{\rm\ref{alg_wolfe_KL}} is well-defined and exhibits the convergence properties of Lemma~{\rm\ref{thm:qian-pan}} under the fulfillment of the corresponding KL conditions therein.
\end{theorem}
\begin{proof}
This follows directly from Theorem~\ref{lemma:wellposedness_of_alg2} and
Lemmas~\ref{thm:qian-pan}, \ref{thm:H1-H2}.
\end{proof}\vspace*{-0.2in}

\section{Coderivative-Based Newton Forward-Backward Method}\label{sec_convex composite}\vspace*{-0.05in}

In this section, we consider a class of convex composite minimization problems, which are {\em first-order nonsmooth} and can incorporate {\em constraints} via set indicator functions. Therefore, the coderivative-based Newton methods CRNM-W and CRNM-WM proposed  and developed in previous sections cannot be applied directly. Following the approach in \cite{khanh2023globally}, implemented there for the case of Armijo's linesearch, we now employ the {\em forward-backward envelope} (FBE) machinery from \cite{patrinos2013proximal} to the novel Wolfe linesearch  in coderivative based Newtonian algorithms. For brevity, our main attention is paid here to applying  CRNM-W to the FBE setting.
We label the new algorithm as the {\em coderivative-based
forward-backward Newton method} (abbr.\ CNFB).

Consider the class of {\em convex composite minimization problem}
\begin{equation}\label{cco}
\mbox{minimize }\;\vp(x)=f(x)+g(x)\;\mbox{ over all }\;x\in\mathbb{R}^n,
\end{equation}
where \(f:\mathbb{R}^n\to\mathbb{R}\) is a \(C^2\)-smooth convex function and \(g:\mathbb{R}^n\to\Rb\) is proper, l.s.c., and convex. The word ``composition" signifies here that the functions $f$ and $g$ have completely different natures. Being extended-real-valued, the function $g$ allows us to implicitly incorporate constraints in the seemingly unconstrained framework of optimization. Problems of type \eqref{cco} arise in many areas of research and practical modeling like machine learning, data science, signal processing, and statistics, where the the nonsmooth term plays a role of regularizators. The main idea of CNFB is applying CRNM-W to the FBE associated with $\vp$ in \eqref{cco}, which happens to be of class $C^{1,1}$ and allows us to eventually solve the original problem.\vspace*{0.03in}

To begin with, recall the relevant definitions and facts needed in what follows. The first constructions are classical; see, e.g., in \cite{rockafellar2009variational}. Given a proper l.s.c.\ function \(\vp:\mathbb{R}^n\to\Rb\) be a parameter \(\gamma>0\) a positive real number.
The {\em Moreau envelope} of \(\vp\) with parameter \(\gamma>0\) and the associated {\em proximal mapping} of $\vp$ are defined, respectively, by
\begin{equation}\label{moreau}
e_\gamma\vp(x):=\inf_{y\in\mathbb{R}^n}\Big\{\vp(y)+\frac{1}{2\gamma}\|y-x\|^2\Big\},\quad x\in\mathbb{R}^n,
\end{equation}
\begin{equation}\label{prox}
\mathrm{Prox}_{\gamma\vp}(x):=\argmin\limits_{y\in\mathbb{R}^n}\Big\{\vp(y)+\frac{1}{2\gamma}\|y-x\|^2\Big\},\quad x\in\mathbb{R}^n.
\end{equation}

Let \(\vp:\mathbb{R}^n\to\Rb=f+g\), where \(f:\mathbb{R}^n\to\mathbb{R}\) is \(C^1\)-smooth and where \(g:\mathbb{R}^n\to\Rb\) is proper l.s.c.\ The {\em forward-backward envelope} (abbr.\ FBE) of \(\vp\) with parameter \(\gamma>0\) is introduced in \cite{patrinos2013proximal} by
\begin{equation}\label{fbe}
\vp_{\gamma}(x):=\inf_{y\in\mathbb{R}^n}\Big\{f(x)+\braket{\nabla f(x),y-x}+g(y)+\frac{1}{2\gamma}\|y-x\|^2\Big\},\quad x\in\mathbb{R}^n.
\end{equation}

The following properties of FBEs are taken from \cite{patrinos2013proximal,stella2017forward}.

\begin{proposition}[\bf properties of FBEs]\label{prop:fbe}
Let \(\vp=f+g\), where \(f:\mathbb{R}^n\to\mathbb{R}\) is convex,
\(C^2\)-smooth, and such that \(\nabla f\) is Lipschitz continuous with modulus \(l>0\), and where \(g:\mathbb{R}^n\to\Rb\) is proper l.s.c., and convex. Then we have the assertions:

{\bf(i)} The FBE \(\vp_{\gamma}\) of \(\vp\) in \eqref{fbe} is \(C^1\)-smooth function with the gradient
\begin{equation*}
\nabla\vp_{\gamma}(x)=\gamma^{-1}\left(I-\gamma \nabla^2 f(x)\right)\left(x-\mathrm{Prox}_{\gamma g}\big(x-\gamma\nabla f(x)\right)\big).
\end{equation*}
Moreover, the set of optimal solutions to problem {\rm(\ref{cco})} coincides with the set of stationary points of the FBE \(\vp_{\gamma}\) for all parameter values \(\gamma\in(0,l^{-1})\).

{\bf(ii)} If \(f(x)=\frac{1}{2}\braket{A,x}+\braket{b,x}+\alpha\) with \(A\in \mathbb{R}^{n\times n}\) being a symmetric and
positive-semidefinite matrix, \(b \in \mathbb{R}^n\), and \(\alpha\in\mathbb{R}\), then for all  \(\gamma\in(0,\frac{1}{l})\), the FBE \(\vp_{\gamma}\) of \(\vp\) is convex and its gradient \(\nabla\vp_{\gamma}\) is Lipschitz continuous with modulus \(L:=2(1-\gamma\lambda_{\min(A)})/\gamma\).  If \(A\) is positive-definite, then \(\vp_{\gamma}\) is strongly convex with modulus \(K:=\min\{(1-\gamma\lambda_{\min(A)})\lambda_{\min(A)},(1-\gamma\lambda_{\max(A)})\lambda_{\max(A)}\}\).
\end{proposition}

Consider further the unconstrained optimization problem
\begin{equation}\label{fbe1}
\mbox{minimize }\;\vp_{\gamma}(x)\;\mbox{ over all }\;x\in\mathbb{R}^n,
\end{equation}
whose stationary points are  optimal solutions to (\ref{cco}) when \(f\) is \(C^2\)-smooth with the Lipschitz continuous gradient. Since the gradient \(\nabla\vp_{\gamma}\) may not be locally Lipschitz continuous, (\ref{fbe1}) is not a \(C^{1,1}\) optimization problem in general. From now on, we assume that \(f\) is a {\em quadratic function} as in
Proposition~\ref{prop:fbe}(ii), and thus \(\nabla\vp_{\gamma}\) is Lipschitz continuous.\vspace*{0.03in}

In what follows, we focus on the convex composite minimization problem
\begin{equation}\label{cco_quadratic}
\mbox{minimize }\;\vp(x)=\frac{1}{2}\braket{Ax,x}+\braket{b,x}+g(x)\;\mbox{ over all }\; x\in\mathbb{R}^n,
\end{equation}
where \(A\in\mathbb{R}^{n\times n}\) is a positive-semidefinite symmetric matrix, \(b\in\mathbb{R}^n\), and where \(g:\mathbb{R}^n\to\Rb\) is a proper l.s.c.\ convex function. To implement GRNM-W for problem \eqref{cco_quadratic}, we need to use the generalized Hessian of $\vp_{\gamma}$, which is calculated in \cite{khanh2023globally}.

\begin{proposition}[\bf calculating of generalized Hessians of FBEs]\label{hes}
Let \(\vp=f+g\) be as in {\rm(\ref{cco_quadratic})}, and let \(\gamma>0\) be such that \(R:=I-\gamma A\succ 0\).
Then the generalized Hessian of \(\vp_{\gamma}\) is calculated by
\begin{equation*}
\begin{aligned}
&\bar{z}\in\partial^2\vp_{\gamma}(\bar{x})(w)\\
\iff
&R^{-1}\bar{z}-Aw\in\partial^2 g\Big(\mathrm{Prox}_{\gamma g}(\bar{u}),\frac{1}{\gamma}(\bar{u}-\mathrm{Prox}_{\gamma g}(\bar{u}))\Big)
(w-\gamma R^{-1}\bar{z})
\end{aligned}
\end{equation*}
for \(\bar{x}\in\mathbb{R}^n,w\in\mathbb{R}^n,\bar{u}:=\bar{x}-\gamma (A\bar{x}+b)\). This can be equivalently expressed as
\begin{equation}\label{2cal}
\partial^2\vp_\gamma(\bar x)(w)=\gamma^{-1}R\big(w+D^*(-\Prox_{\gamma g})(\bar{u})(Rw)\big).
\end{equation}
\end{proposition}

We employ \eqref{2cal} in the construction of CNFB to solve \eqref{cco_quadratic}.

\begin{algorithm}
\caption{Coderivative-based Newton forward-backward method (CNFB)}\label{alg_CNFB}
{\bf Input:} \(x^0 \in \mathbb{R}^n\), $\gamma>0$ such that $R:=I-\gamma A\succ 0$, \(0 < \sigma_1 < \sigma_2 <1\), \(c>0\), \(\rho \in (0,1]\).
\begin{algorithmic}[1]
\For{\(k=0,1,\ldots\)}
\State If \(\nabla \varphi_\gamma (x^k)=0\), stop; otherwise set $u^k:=x^k-\gamma(Ax^k+b)$, $v^k:=\Prox_{\gamma g}(u^k)$, \(\mu_k = c\|\nabla \varphi_\gamma(x^k)\|^{\rho}=c\gamma^{-1}\|R(x^k-v^k)\|^\rho\) and go to next step.
\State Choose $B_k \succ 0$. Find \(d^k \in \mathbb{R}^n\) such that \(-\nabla \varphi_\gamma (x^k) \in \partial^2 \varphi_\gamma (x^k){(d^k)} + \mu_k B_k d^k \), i.e.,
$$
\gamma^{-1}R(v^k-x^k) \in \gamma^{-1}R\big(d^k+D^*(-\Prox_{\gamma g})(u^k)(Rd^k)\big) + \mu_k B_k d^k.
$$
\State Set \(\tau_k=1\) and check the Wolfe conditions:
\[\varphi_\gamma(x^k+\tau_k d^k) \le \varphi_\gamma(x^k)+\sigma_1 \tau_k \braket{\nabla\varphi_\gamma(x^k),d^k},\]
\[\braket{\nabla \varphi_\gamma(x^k+\tau_k d^k),d^k} \ge \sigma_2 \braket{\nabla \varphi_\gamma(x^k),d^k}.\]
If the latter conditions do not hold, adjust \(\tau_k\) (using any specific implementation of the Wolfe linesearch) until it satisfies the Wolfe conditions. We assume that in the implementation there exists an upper bound \(\tau_{\mathrm{max}}\) on the maximum stepsize allowed.
\State Set \(x^{k+1}=x^k+\tau_kd^k\).
\EndFor
\end{algorithmic}
\end{algorithm}

To proceed with establishing the convergence properties of Algorithm~\ref{alg_CNFB}, we need to recall the two results obtained in \cite{khanh2023globally}. The first one is taken from \cite[Proposition~5]{khanh2023globally}.

\begin{proposition}[\bf metric regularity of FBEs]\label{proposition: metric} Let \(\vp=f+g\) be as in {\rm(\ref{cco_quadratic})}, and let \(\gamma>0\) be such that \(R:=I-\gamma A\succ 0\). For any \(\bar{x}\in\mathbb{R}^n\) with \(0\in\partial\vp(\bar{x})\), the following assertions hold:

{\bf(i)} \(\partial\vp\) is metrically regular around \((\bar{x},0)\) if and only if \(\nabla\vp_\gamma\) is metrically regular around \(\bar{x}\).

{\bf(ii)} \(\bar{x}\) is a tilt-stable local minimizer of \(\vp\) if and only if \(\bar{x}\) is a tilt-stable local minimizer of \(\vp_\gamma\).

{\bf(iii)} \(\|\partial^2\vp_\gamma(\bar{x})^{-1}\|\le\|\partial^2\vp(\bar{x},0)^{-1}\|+\gamma\|R^{-1}\|\).
\end{proposition}

In the next proposition taken from \cite[Propositions~6]{khanh2023globally}, we use the notion of {\em twice epi-differentiability} of extended-real-valued functions, which is studied  and applied to optimization in \cite{rockafellar2009variational} and  more recent papers; see, e.g., \cite{mms1,mms,ms}, where the reader can find more details and references.

\begin{proposition}[\bf semismoothness$^*$ and directional differentiability for FBEs]\label{proposition: semismooth} In the setting of Proposition~{\rm\ref{proposition: metric}}, the following assertions hold:

{\bf(i)} \(\nabla\vp_{\gamma}\) is semismooth\({}^*\) at \(\bar{x}\) if \(\partial g\) is semismooth\({}^*\) at \((\bar{x},\bar{v})\), where \(\bar{v}:=-A\bar{x}-b\);

{\bf(ii)} \(\nabla\vp_{\gamma}\) is directionally differentiable at \(\bar{x}\) if \(g\) is twice epi-differentiable at \(\bar{x}\) for \(\bar{v}\).
\end{proposition}

Now we are ready to establish comprehensive convergence results for CNFB.

\begin{theorem}[\bf performance of CNFB]\label{thm: convergence of GRNM for cco} Consider the convex composite minimization problem \eqref{cco_quadratic}, where the symmetric matrix \(A\) is
positive-semidefinite. Then we have the assertions:

{\bf(i)} CNFB $($Algorithm~{\rm\ref{alg_CNFB})} either stops after finitely many iterations at a minimizer of \(\vp\), or generates a sequence \(\{x^k\}\) whose accumulation points are optimal solutions to problem {\rm(\ref{cco_quadratic})}.

{\bf (ii)} If the subgradient mapping \(\partial\vp\) is metrically regular around \((\bar{x},0)\) with modulus \(\kappa>0\) $($which holds when \(A\) is positive-definite$)$, where \(\bar{x}\) is an accumulation point of \(\{x^k\}\),
then the sequence \(\{x^k\}\) converges with local R-linear rate to \(\bar{x}\), and \(\bar{x}\) is a tilt-stable local minimizer of \(\vp\).

{\bf(iii)} The local convergence rate of \(\{x^k\}\) is Q-superlinear if the subgradient mapping \(\partial g\) is semismooth\({}^*\) at \((\bar{x},\bar{v})\), where \(\bar{v}:=-A\bar{x}-b\), and if either one of the following conditions holds:

{\bf(a)} \(g\) is twice epi-differentiable at \(\bar{x}\) for \(\bar{v}\).

{\bf(b)} The linesearch constants satisfy the conditions \(\sigma_1\in(0,\frac{1}{2LK})\) and \(\sigma_2\in(1-\frac{K}{L},1)\), where \(L:=2(1-\gamma\lambda_{\min(A)})/\gamma\) and \(K:=\kappa+\gamma\|B^{-1}\|\).
\end{theorem}
\begin{proof}
By Proposition~\ref{prop:fbe}(i), minimizing \(\vp\) reduces to minimizing the FBE function \(\vp_\gamma\) of class $C^{1,1}$ when the parameter $\gamma>0$ is sufficiently small. We now verify each claim of the theorem.

(i) Proposition~\ref{prop:fbe}(ii) tells us that the FBE $\vp_\gamma$ is convex and its gradient \(\nabla\vp_\gamma\) is Lipschitz continuous with modulus \(L:=2(1-\gamma\lambda_{\min(A)})/\gamma\). By \cite[Theorem~3.2]{chieu2011characterizing}, the generalized Hessian \(\partial^2\vp_\gamma(x)\) is positive-semidefinite for all \(x\in\mathbb{R}^n\). Then GRNM-W is well-defined, and the claimed assertion (i) follows from
Theorems~\ref{thm:wellposedness} and \ref{thm:weakconvergence_wolfe}.

(ii) By \cite[Proposition 4.5]{drusvyatskiy2014second}, the tilt-stability of \(\vp\) at \(\bar{x}\) with modulus \(\kappa\) follows from the metric regularity of \(\partial \vp\) and the convexity of \(\vp\). By Proposition~\ref{proposition: metric} (i), the gradient mapping $\nabla\vp_\gamma$ is metrically regular around $\bar x$. Then the
R-linear convergence of \(\{x^k\}\) follows from
Theorem~\ref{thm:convergence_wolfe}(ii). We also need to show that if $A$ is positive-definite, then $\partial\vp$ is metrically regular around $(\bar x,0)$. Indeed, the positive-definiteness of \(A\) ensures
by Proposition~\ref{prop:fbe}(ii) that \(\vp_\gamma\) is strongly convex with the Lipschitz continuous gradient. Using \cite[Theorem~5.1]{chieu2011characterizing} tells us that \(\partial^2\vp_\gamma(x)\) is positive-definite for all \(x\in\mathbb{R}^n\). This implies by \cite[Proposition~4.5 and Theorem~4.6]{drusvyatskiy2014second} that \(\nabla\vp_\gamma\) is metrically regular around \(\bar{x}\). Therefore, we deduce from Proposition~\ref{proposition: metric}(i) that $\partial\vp$ is metrically regular around $(\bar x,0)$.

(iii) Proposition~\ref{proposition: semismooth}(i) verifies that \(\nabla\vp_\gamma\) is semismooth\({}^*\) at \(\bar{x}\). In case (a), \(\nabla\vp_\gamma\) is directionally differentiable at \(\bar{x}\) by Proposition~\ref{proposition: semismooth}(ii). Then
Theorem~\ref{thm:convergence_wolfe}(iii,a) yields assertion (iii) in this case.

To complete the proof of the theorem, it remains to consider case (b) in (iii). It follows from Proposition~\ref{prop:fbe}(ii) that the number \(L\) therein is a Lipschitz constant of \(\nabla\vp_\gamma\) around \(\bar{x}\). Then Proposition~\ref{proposition: metric}(i) and the Mordukhovich criterion from \cite[Theorem~9.40]{rockafellar2009variational} (see \cite{m93}) ensure that the FBE $\vp_\gamma$ is metrically regular around \(\bar{x}\) with the modulus \(K\) defined above, and thus the claimed assertion (iii) holds by Theorem~\ref{thm:convergence_wolfe}(iii,b).
\end{proof}\vspace*{-0.23in}

\section{Coderivative-Based Newton Augmented Lagrangian Method}\label{sec_CNAL}\vspace*{-0.05in}

In this section, we develop another algorithm for a class of convex composite minimization problems, first-order nonsmooth and constrained, by embedding GRNM-W into the augmented Lagrangian method. Inspired by the semismooth Newton augmented Lagrangian method (SSNAL) \cite{li2018highly}, an augmented Lagrangian method employing the semismooth Newton method as inner problem solver, we propose here the new {\em coderivative-based Newton augmented Lagrangian method} (abbr.\ CNAL).

Consider the following {\em linear-convex composite minimization problem} written in the unconstrained extended-real-valued format:
\begin{equation}\label{problem:primal}
({\bf P})\;\mbox{ minimize }\;f(x):=h(\mathcal A x)-\braket{c,x}+p(x)\;\mbox{ over all }\;x\in\mathbb{R}^n,
\end{equation}
where $\mathcal A:\mathbb{R}^n\to\mathbb{R}^m$ is an $m\times n$ matrix,  $h:\mathbb{R}^m\to\R$ is a l.s.c.\ convex function, $p:\mathbb{R}^n
\to{\Rb}$ is a proper l.s.c.\ convex function, and $c\in\mathbb{R}^n$. The {\em dual problem} of \eqref{problem:primal} is given by
\begin{equation}\label{problem:dual}
\begin{aligned}
({\bf D})\;\mbox{ minimize }\; h^*(y)+p^*(z)\\
\mbox{ subject to }\quad\,\, \mathcal A^* y + z = c,
\end{aligned}
\end{equation}
where $h^*$ and $p^*$ are the Fenchel conjugates  of $h$ and $p$, respectively, and where $\mathcal A^*={\cal A}^{T}$ for the adjoint/transpose matrix of $\mathcal{A}$; see, e.g., \cite[Example~11.41]{rockafellar2009variational}.

Given $\sigma>0$, the {\em augmented Lagrangian} associated with \eqref{problem:dual} is
\begin{equation*}
 \mathcal L_{\sigma}(y,z;x):=h^*(y)+p^*(z)-\braket{x,\mathcal A^*y+z-c}+\frac{\sigma}{2}\|\mathcal A^*y+z-c\|^2
\end{equation*}
whenever $(y,z,x)\in\mathbb{R}^m\times\mathbb{R}^n\times\mathbb{R}^n$. From now on, we impose the additional assumption as follows.

\begin{assumption}\label{assumption:CNAL}
The function $h$ in \eqref{problem:primal} is strongly convex and
$C^2$-smooth.
\end{assumption}

We propose the following coderivative-based Newton augmented Lagrangian method (CNAL) to solve the dual problem \eqref{problem:dual} and thus the primal problem \eqref{problem:primal} by strong duality.

\begin{algorithm}
\caption{Coderivative-based Newton augmented Lagrangian method (CNAL)}\label{alg_CNAL}
{\bf Input:} $\sigma_0 >0$, $(y^0,z^0,x^0)\in\mathbb{R}^m\times\textup{dom}\, p^*\times\mathbb{R}^n$.
\begin{algorithmic}[1]
\For{\(k=0,1,\ldots\)}
\State Compute
\begin{equation}\label{p2:alm-sub}
(y^{k+1},z^{k+1}) \approx \arg\min\big\{\Psi_k (y,z):= \mathcal L_{\sigma_k}(y,z;x^k)\big\}
\end{equation}
via the coderivative-based Newton method GRNM-W. The stopping criterion is:
$\Psi_k(y^{k+1},z^{k+1})-\inf\Psi_k\le\epsilon^2_k/2\sigma_k$ with $\sum_{k=0}^\infty\epsilon_k<\infty$.
\State Compute $x^{k+1} = x^k - \sigma_k(\mathcal A^*y^{k+1} + z^{k+1} - c)$ and update $\sigma_{k+1} \uparrow \sigma_\infty\leq \infty$ .
\EndFor
\end{algorithmic}
\end{algorithm}

The next theorem establishes convergence properties of iterates in
Algorithm~\ref{alg_CNAL}.

\begin{theorem}[\bf convergence of CNAL]\label{thm:CNAL} Suppose that the primal problem \eqref{problem:primal} admits an optimal solution and that all the assumptions imposed above are satisfied. Then for any infinite sequence of iterates $\{(y^k,z^k,x^k)\}$ generated by Algorithm~{\rm\ref{alg_CNAL}}, we have that $\{x^k\}$ converges to an optimal solution of problem \eqref{problem:primal} while $\{(y^k,z^k)\}$ converges to an optimal solution of the dual problem \eqref{problem:dual}.
\end{theorem}
\begin{proof} Observe first that Assumption~\ref{assumption:CNAL} ensures by \cite[Proposition~12.60]{rockafellar2009variational} that $h^*$ is also a strongly convex and $C^2$-smooth function. It follows from Fenchel’s duality theorem \cite[Corollary~31.2.1]{rockafellar1970convex} that strong duality holds in the setting of \eqref{problem:primal} and \eqref{problem:dual}. Employing finally \cite[Theorem~4]{rockafellar1976augmented} verifies the claimed convergence results.
\end{proof}

Next we consider the subproblem in \eqref{p2:alm-sub} formulated as
\begin{equation*}
\mbox{minimize }\;\Psi(y,z):=\mathcal L_{\sigma}(y,z;\tilde{x})\;\mbox{ over }\;(y,z)\in\mathbb{R}^m\times\mathbb{R}^n
\end{equation*}
for which the optimal solution is given by
\begin{equation*}
\bar y = \argmin\psi(y), \quad \bar z = \Prox_{p^*/\sigma}(\tilde{x}/\sigma-\mathcal A^*\bar y+c)
\end{equation*}
via the proximal mapping \eqref{prox}, where the function $\psi$ is defined by
\begin{equation}\label{psi}
\begin{aligned}
\psi(y):=\inf_z \Psi(y,z)=h^*(y)+ e_{\frac{1}{\sigma}}p^*\big(\tilde{x}/\sigma-(\mathcal A^*y-c)\big)-\frac{1}{2\sigma}\|\tilde{x}\|^2
\end{aligned}
\end{equation}
via the Moreau envelope \eqref{moreau}. To solve subproblem \eqref{p2:alm-sub}, we apply GRNM-W to the minimization of $\psi$. The gradient of $\psi$ is computed by
\begin{equation*}
\begin{aligned}
\nabla\psi(y)&=\nabla h^*(y) - \mathcal A \nabla e_{\frac{1}{\sigma}}p^*\big(\tilde{x}/\sigma-(\mathcal A^*y-c)\big)\\
&= \nabla h^*(y) - \mathcal A\big(\sigma (u'-\Prox_{\frac{1}{\sigma}p^*}(u'))\big)\\
&= \nabla h^*(y) - \mathcal{A}\big(\sigma (\sigma^{-1}\Prox_{\sigma p}(\sigma u'))\big)\\
&= \nabla h^*(y) - \mathcal A \Prox_{\sigma p}\big(\tilde{x}-\sigma(\mathcal A^*y-c)\big),\\
\end{aligned}
\end{equation*}
where \(u':=\tilde{x}/\sigma-(\mathcal A^*y-c)\). Note that $\nabla\psi$ is locally Lipschitz continuous while being nonsmooth due to the presence of the proximal mapping.

To implement Algorithm~\ref{alg_CNAL}, we need to constructively evaluate the generalized Hessian of $\psi$ (i.e., the limiting coderivative of $\nabla\psi$) in Step~2. The following two lemmas of their own interest are useful to estimate the required second-order construction.

\begin{lemma}[\bf limiting coderivatives of special compositions]\label{prop:AS} Consider the composition \(G: = \mathcal A \circ S\), where \(S:\mathbb{R}^d\to\mathbb{R}^n\) is Lipschitz continuous around $\bar x\in\mathbb{R}^d$, and where $\mathcal A:\mathbb{R}^n\to\mathbb{R}^m$ is linear. Then we have the equality
\begin{equation}\label{eq:DAS}
 D^* G(\bar x)(v) = D^* S(\bar x)(\mathcal A^*v)\;\mbox{ for all }\;v\in\mathbb{R}^m.
\end{equation}
\end{lemma}
\begin{proof} We aim at checking the equality
\begin{equation}\label{reg-equ}
\widehat{D}^* G(\bar x)(v) = \widehat D^* S(\bar x)(\mathcal A^* v),\quad v\in\mathbb{R}^m,
\end{equation}
for the regular coderivative \eqref{reg-cod}, which yields \eqref{eq:DAS} by definition \eqref{lim-cod} and the arbitrary choice of $\bar x$. To proceed with verifying \eqref{reg-equ}, observe that $u\in\widehat{D}^* G(\bar x)(v)=\widehat{D}^*(\mathcal A\circ S)(\bar x)(v)$ gives us $(u,-v)\in\widehat N_{\gph \mathcal A\circ S}(\bar x, \mathcal A(S(\bar x)))$, which means by \eqref{rnc} that
\begin{equation}\label{eq:G}
\limsup_{(x,\mathcal A(S(x)))\to (\bar x, \mathcal A(S(\bar x)))}\frac{\braket{u,x-\bar x}-\braket{v,\mathcal A(S(x))-\mathcal A(S(\bar x))}}{\|x-\bar x\|+\|\mathcal A(S(x))-\mathcal A(S(\bar x))\|} \le 0.
\end{equation}
On the other hand, we have by  $u\in\widehat D^* S(\bar x)(\mathcal A^* v)$ that \((u,-\mathcal A^*v)\in \widehat N_{\gph\,S}(\bar x,S(\bar x))\), i.e.,
\begin{equation}\label{eq:AS}
\limsup_{(x,\mathcal S(x))\to (\bar x, \mathcal S(\bar x))}\frac{\braket{u,x-\bar x}-\braket{\mathcal A^*v,S(x)-S(\bar x)}}{\|x-\bar x\|+\|S(x)-S(\bar x)\|} \le 0.
\end{equation}
Let us show that \eqref{eq:G} and \eqref{eq:AS} are equivalent. Indeed, by
$$
\|x-\bar x\|+\|\mathcal AS(x)-\mathcal AS(\bar x)\| \le \|x-\bar x\|+\|\mathcal A\|\cdot\|S(x)-S(\bar x)\|\le C (\|x-\bar x\|+\|S(x)-S(\bar x)\|)
$$
with $C:=\max\{1,\|\mathcal A\|\}$, it follows from  \eqref{eq:G} that
\begin{equation*}
\begin{aligned}
&\limsup_{(x,\mathcal S(x))\to (\bar x, \mathcal S(\bar x))}\frac{\braket{u,x-\bar x}-\braket{\mathcal A^*v,S(x)-S(\bar x)}}{\|x-\bar x\|+\|S(x)-S(\bar x)\|}\\
=&\limsup_{(x,\mathcal S(x))\to (\bar x, \mathcal S(\bar x))}\frac{\braket{u,x-\bar x}-\braket{v,\mathcal A S(x)-\mathcal AS(\bar x)}}{\|x-\bar x\|+\|S(x)-S(\bar x)\|}\\
\le&\limsup_{(x,\mathcal AS(x))\to (\bar x, \mathcal AS(\bar x))} C \frac{\braket{u,x-\bar x}-\braket{v,\mathcal AS(x)-\mathcal AS(\bar x)}}{\|x-\bar x\|+\|\mathcal AS(x)-\mathcal AS(\bar x)\|}\le 0,
\end{aligned}
\end{equation*}
which tells us that \eqref{eq:AS} is satisfied. Conversely, if \eqref{eq:AS} holds, then the Lipschitz continuity of $S$ around $\bar x$with some modulus $L\ge 0$ yields
$$
\|x-\bar x\|+\|S(x)-S(\bar x)\|\le (1+L)\|x-\bar x\|\le (1+L)(\|x-\bar x\|+\|\mathcal AS(x)-\mathcal AS(\bar x)\|),
$$
which ensures in turn the relationships
\begin{equation*}
\begin{aligned}
&\limsup_{(x,\mathcal AS(x))\to (\bar x, \mathcal AS(\bar x))}
\frac{\braket{u,x-\bar x}-\braket{v,\mathcal AS(x)-\mathcal AS(\bar x)}}{\|x-\bar x\|+\|\mathcal AS(x)-\mathcal AS(\bar x)\|}\\
=&\limsup_{(x,\mathcal AS(x))\to (\bar x, \mathcal AS(\bar x))}
\frac{\braket{u,x-\bar x}-\braket{\mathcal A^*v,S(x)-S(\bar x)}}{\|x-\bar x\|+\|\mathcal AS(x)-\mathcal AS(\bar x)\|}\\
\le&\limsup_{(x,\mathcal S(x))\to (\bar x, \mathcal S(\bar x))}
(1+L)\frac{\braket{u,x-\bar x}-\braket{\mathcal A^*v,S(x)-S(\bar x)}}{\|x-\bar x\|+\|S(x)-S(\bar x)\|}\le 0.
\end{aligned}
\end{equation*}
This tells us that \eqref{eq:G} is satisfied and thus verifies the claim.
\end{proof}

\begin{lemma}[\bf regular coderivatives of special compositions]\label{prop:SA}
Let $F=S\circ\mathcal B$, where $S:\mathcal\mathbb{R}^m\to\mathbb{R}^q$ is a continuous mapping, and where $\mathcal B:\mathbb{R}^n\to\mathbb{R}^m$ is an affine mapping given by $\mathcal Bx:=\mathcal Ax+b$ with a linear operator$\mathcal A:\mathbb{R}^n\to\mathbb{R}^m$ and $b\in\R^m$. Then we have the inclusion
\begin{equation}\label{reg-inc}
 \mathcal A^* \widehat{D}^*S(\mathcal A\bar x + b)(v) \subset \widehat{D}^*F(\bar x)(v)\;\mbox{ for all }\;v\in\mathbb{R}^q.
\end{equation}
\end{lemma}
\begin{proof} To verify \eqref{reg-inc}, pick $u\in\mathcal A^* \widehat{D}^*S(\mathcal A\bar x+b)(v)$ and get $u=\mathcal A^*w$ with $(w,-v)\in \widehat N_{\gph S}(\mathcal A\bar x+b, S(\mathcal A\bar x+b))$. It follows from definition \eqref{rnc} that
\begin{equation*}
\limsup_{(y,S(y))\to(\mathcal A\bar x+b, S(\mathcal A\bar x+b))}
\frac{\braket{w,y-\mathcal A\bar x-b}-\braket{v,S(y)-S(\mathcal A \bar x+b)}}{\|y-\mathcal A\bar x-b\|+\|S(y)-S(\mathcal A \bar x+b)\|}\le 0.
\end{equation*}
Taking now  $y=Ax+b$ tells us that
\begin{equation*}
\limsup_{(\mathcal Ax+b,S(\mathcal Ax+b))\to(\mathcal A\bar x+b, S(\mathcal A\bar x+b))}
 \frac{\braket{\mathcal A^*w,x-\bar x}-\braket{v,S(\mathcal Ax+b)-S(\mathcal A \bar x+b)}}{\|\mathcal Ax-\mathcal A\bar x\|+\|S(\mathcal Ax+b)-S(\mathcal A \bar x+b)\|}\le 0.
\end{equation*}
Due to the obvious estimates
$$
\|\mathcal Ax-\mathcal A\bar x\|+\|S(\mathcal Ax+b)-S(\mathcal A\bar x+b)\| \le \|\mathcal A\|\cdot\|x-\bar x\|+\|S(\mathcal Ax+b)-S(\mathcal A\bar x+b)\|\le C\big(\|x-\bar x\|+\|S(x)-S(\bar x)\|\big)
$$
with $C=\max\{1,\|\mathcal A\|\}$, we deduce that
\begin{equation*}
\begin{aligned}
&\limsup_{(x,S(\mathcal Ax+b))\to(\bar x, S(\mathcal A\bar x+b))}
\frac{\braket{u,x-\bar x}-\braket{v,S(\mathcal Ax+b)-S(\mathcal A \bar x+b)}}{\|x-\bar x\|+\|S(\mathcal Ax+b)-S(\mathcal A \bar x+b)\|}\\
\le&\limsup_{(x,S(\mathcal Ax+b))\to(\bar x, S(\mathcal A\bar x+b))}
C\frac{\braket{u,x-\bar x}-\braket{v,S(\mathcal Ax+b)-S(\mathcal A \bar x+b)}}{\|\mathcal A x-\mathcal A \bar x\|+\|S(\mathcal Ax+b)-S(\mathcal A \bar x+b)\|}\\
\le&\limsup_{(\mathcal Ax+b,S(\mathcal Ax+b))\to(\mathcal A\bar x+b, S(\mathcal A\bar x+b))}
C\frac{\braket{u,x-\bar x}-\braket{v,S(\mathcal Ax+b)-S(\mathcal A \bar x+b)}}{\|\mathcal Ax-\mathcal A\bar x\|+\|S(\mathcal Ax+b)-S(\mathcal A \bar x+b)\|}\\
=&\limsup_{(\mathcal Ax+b,S(\mathcal Ax+b))\to(\mathcal A\bar x+b, S(\mathcal A\bar x+b))}
C\frac{\braket{\mathcal A^*w,x-\bar x}-\braket{v,S(\mathcal Ax+b)-S(\mathcal A \bar x+b)}}{\|\mathcal Ax-\mathcal A\bar x\|+\|S(\mathcal Ax+b)-S(\mathcal A \bar x+b)\|}\le 0,
\end{aligned}
\end{equation*}
which shows that $u\in \widehat{D}^*F(\bar x)(v)$ and thus completes the proof.
\end{proof}

Finally in this section, we obtain a lower estimate of the generalized Hessian $\partial^2\psi$, which is sufficient for computational purposes in what follows.

\begin{theorem}[\bf lower estimates for generalized Hessians]\label{2lower}
The generalized Hessian of $\psi$ from \eqref{psi} admits the following lower estimate, where $u=\tilde{x}-\sigma(\mathcal A^*y-c)$:
\begin{equation*}
\begin{aligned}
\partial^2\psi(y)(w)\,\supset\,\nabla^2 h^*(y)(w) - \sigma\mathcal A\big(\widehat D^*\Prox_{\sigma p}\big)(u)(-\mathcal A^* w),\quad w\in\mathbb{R}^m.
\end{aligned}
\end{equation*}
\end{theorem}
\begin{proof}
Note that the standing Assumption~\ref{assumption:CNAL} implies that $h^*$ is $C^2$-smooth. It follows from the coderivative sum rule in \cite[Theorem~3.9]{mordukhovich2018variational} that
$$
\big(D^*\nabla\psi\big)(y)(w)=\nabla^2h^*(y)(w)+\big(D^*(-\mathcal{A}S)\big)(y)(w),
$$
where $S(y):=\Prox_{\sigma p}(\tilde{x}-\sigma(\mathcal A^*y-c))$.
By Lemma \ref{prop:AS}, we have $\big(D^*(-\mathcal A S)\big)(y)(w)=D^*S(y)(-\mathcal A^*w)$.
Finally, we deduce from Lemma~\ref{prop:SA} that
$$
D^*S(y)(-\mathcal A^*w)\supset \widehat D^*S(y)(-\mathcal A^*w)\supset -\sigma\mathcal{A}\big(\widehat D^*\Prox_{\sigma p}\big)(u)(-\mathcal A^*w),
$$
which therefore justifies the claimed formula.
\end{proof}\vspace*{-0.2in}

\section{Applications of CNFB to Support Vector Machines}\label{sec_svm}\vspace*{-0.05in}

In this section, we apply the proposed CNFB to optimization problems in support vector machines, which can be formulated as a quadratic programming problem with some special structure. The constructive implementations of CNFB to solve such problems with conducting numerical experiments require {\em explicit calculations} of the {\em generalized Hessians} for the functions in question.\vspace*{-0.1in}

\subsection{Support Vector Machines}\label{subsec:svm}

A {\em support vector machine} (SVM) is a machine learning model for binary classification. Given a training set \(\{\mathbf{x}_i, y_i\}_{i=1}^n\) with \(\mathbf{x}_i\in\mathbb{R}^d\) and \(y_i\in\{1,-1\}\), the aim of SVMs is to select an appropriate class of {\em classifiers} for the training data and to optimize its characteristics under the imposed requirements; see \cite[Chapter~6]{svm} with the references therein.

For the class of linear classifiers, the optimization problem in SVMs is formulated as follows:
\begin{equation}\label{svm-primal}
\begin{aligned}
\mbox{minimize }\;\quad&\frac{1}{2}\|\mathbf{w}\|^2 + C\sum_{i=1}^n \xi_i\;\mbox{ over }\;\mathbf{w},b,\mathbf{\xi}\\
\mbox{subject to }\quad &y_i(\mathbf{w}^T\mathbf{x}_i+b)\ge 1-\xi_i,\quad \xi_i\ge 0\;\mbox{ for }\;i=1,\ldots,n.
\end{aligned}
\end{equation}
The dual problem of \eqref{svm-primal} is defined by
\begin{equation}\label{svm-dual}
\begin{aligned}
\mbox{minimize }\;\quad &\frac{1}{2}\mathbf{\lambda}^T D\mathbf\lambda-\mathbf{\lambda}^T\mathbf{e}\;\mbox{ over }\;{\mathbf\lambda}\in\mathbb{R}^n\\
\mbox{subject to }\;\quad &\mathbf{\lambda}^T\mathbf{y}=0,\quad\mathbf 0 \le \mathbf{\lambda}\le C\mathbf{e},
\end{aligned}
\end{equation}
where $\mathbf e\in\R^n$ is the all-one vector, and where $D\in\R^{n\times n}$ is a positive-semidefinite symmetric matrix with \(D_{ij}=y_iy_j\mathbf{x}_i^T\mathbf{x}_j\). We can see that \eqref{svm-dual} is a convex quadratic programming problem with a single linear constraint and bound constraints on the variables (i.e., the feasible region is the intersection of a hyperplane and a box). Note that for nonlinear classifiers in SVMs, the dual optimization problem can also be written in such a form by using a kernel function.

Having this in mind, we consider below the following general class of {\em convex quadratic programs} with a {\em single linear constraint} and {\em bound constraints} on the variables (abbr.\ SLBQP):
\begin{equation}\label{eq:SLBQP}
\begin{aligned}
\mbox{minimize }\;\quad &\frac{1}{2}x^TQx+c^Tx\;\mbox{ over }\;x\in\mathbb{R}^n\\
\mbox{subject to }\;\quad& a^Tx=b,\quad l\le x \le u,
\end{aligned}
\end{equation}
where $Q\in\R^{n\times n}$ is a positive-semidefinite symmetric matrix, $c\in\R^n$, $a\in\R^n$, $b\in\R$, $l\in (\R\cup\{-\infty\})^n$, and $u\in (\R\cup\{\infty\})^n$. This class includes support vector machine and bound-constrained quadratic programming problems as special cases.

Observe that the quadratic program \eqref{eq:SLBQP} is a special case of the convex composite minimization problems of type \eqref{cco_quadratic} with $\vp(x):=f(x)+g(x)$, $f(x):=\frac{1}{2}x^TQx+c^Tx$, $g(x):=\delta_\Gamma(x)$, where $\delta_\Gamma$ is the indicator function of the feasible set $\Gamma:=\{x\in\R^n\;|\;a^Tx=b,\;l\le x \le u\}$. To apply the proposed CNFB to solving problem \eqref{eq:SLBQP}, we need to explicitly calculate the generalized Hessian of FBE, which reduces to calculating the coderivative of the (minus) proximal mapping (see Step~3 of Algorithm~\ref{alg_CNFB}) and eventually to calculating the coderivative of the (minus) projection operator $\Proj_\Gamma$. This goal is achieved  in the next subsection.\vspace*{-0.1in}

\subsection{Coderivatives of Polyhedral Projectors}

Let $\Gamma$ be a nonempty convex polyhedral set given by
\begin{equation}\label{Gamma}
\begin{aligned}
\Gamma&=\big\{x\in \R^n\;\big|\;\braket{a_i,x}=b_i\;\mbox{ for }\;i\in R,\quad\braket{c_i,x}\le d_i\;\mbox{ for }\;i\in S\big\}\\
&=\big\{x\in \R^n\;\big|\; A^Tx=b,\;C^Tx\le d\big\},
 \end{aligned}
\end{equation}
where $b\in\R^{|R|\times 1}$, $d\in\R^{\mid S \mid\times 1}$, and where the matrices $A^T\in\R^{|R|\times n}$ and $C^T\in\R^{|S|\times n}$ are formed by the row vectors $a_i^T$ as $i\in R$ and $c_i^T$ as $i\in S$,
respectively.\vspace*{0.03in}

The following result is taken from \cite[Corollary~4.3]{Henrion2010}.

\begin{lemma}[\bf coderivative calculations for normals to polyhedra]\label{thm:coderivative-normal-cone} Given a convex polyhedron $\Gamma$ in \eqref{Gamma}, define the family of index sets
\begin{equation*}
\mathcal I_\Gamma:=\big\{S'\subset S\;\big|\;\exists\,x\in\Gamma\;\text{ such that }\;\braket{c_i,x}=d_i,\,i\in S'\;\text{ and }\;\braket{c_i,x}<d_i,\,i\in S\backslash S'\big\}
\end{equation*}
and pick any $\bar x^*\in N_\Gamma(\bar x)$\;\mbox{ with }\;$\bar x\in\Gamma$.  Then we have $z\in D^* N_\Gamma(\bar x;\bar x^*)(w)$ if and only if
\begin{equation}\label{cod-cal}
\left\{\begin{aligned}
&-w \in B_{J,K}:=\big\{x\;\big|\;\braket{a_i,x}=0,\,i\in R;\,\,\braket{c_i,x}=0,\,i\in J;\,\,\braket{c_i,x}\le 0,\,i\in K\big\},\\
&z \in A_{J,K}:=B_{J,K}^*=\mathrm{span}\big\{a_i,\,i\in R\big\}+\mathrm{span}\big\{c_i,\,i\in J\big\}+\mathrm{cone}\big\{c_i,\, i\in K\backslash J\big\},
\end{aligned}\right.
\end{equation}
where $K\in\mathcal I_\Gamma$, $J\subset K \subset I(\bar x):=\big\{i\in S\;\big|\;\braket{c_i,\bar x}=d_i\big\}$, $\bar{x}^*\in \mathrm{span}\big\{a_i,\,i\in R\big\}+\mathrm{cone}\big\{c_i,\,i\in J\big\}$.
\end{lemma}

The next lemma follows from the fact that for any convex set $\Gamma$, we have $\Proj_\Gamma = (I+N_\Gamma)^{-1}$.

\begin{lemma}[\bf coderivatives of projections and normal cone mappings]\label{cod-proj} Let $u\in\R^n$. Then the inclusion
$w\in D^*(-\Proj_\Gamma)(u)(d)$ is equivalent to
$$
w+d\in D^*N_\Gamma\big(\Proj_\Gamma(u),u-\Proj_\Gamma(u)\big)(-w).
$$
\end{lemma}

Combining the two lemmas above, we arrive at the explicit formula to calculate coderivatives of (minus) projection operators.

\begin{theorem}[\bf calculating coderivatives of polyhedral projections]\label{thm:cod-proj} We have
$w\in D^*(-\Proj_\Gamma)(u)(d)$ if and only if $w\in B_{J,K}$ and $w+d\in A_{J,K}$, where all the data are taken from Lemma~{\rm\ref{thm:coderivative-normal-cone}}.
\end{theorem}

It is computationally convenient to extract a {\em linear mapping} from the above coderivative descriptions. It can be done, e.g., by selecting the data as follows:
\[J=K=I\big(\Proj_\Gamma(u)\big)=\big\{i\in S\;\big|\;\braket{c_i,\Proj_\Gamma(u)}=d_i\big\}.\]
Then we get from \eqref{cod-cal} the description
\begin{equation*}
\left\{\begin{aligned}
&w \in B_{J,J}=\big\{x\;\big|\;\braket{a_i,x}=0,\;i\in R;\,\braket{c_i,x}=0,\;i\in J\big\},\\
&w+d \in A_{J,J}=\mathrm{span}\big\{a_i,\;i\in R\big\}+\mathrm{span}\big\{c_i,\;i\in J\big\}=B_{J,J}^\perp,
\end{aligned}\right.
\end{equation*}
which is equivalent to the simple inclusions
\begin{equation*}
\begin{aligned}&d \in -w+B_{J,J}^\perp,
&-w \in B_{J,J}.
\end{aligned}
\end{equation*}
Therefore, we arrive at the projection expression
\[-w=\Proj_{B_{J,J}}(d)\Longleftrightarrow w=-\Proj_{B_{J,J}}(d),\]
which finally brings us to the exact explicit formula
\begin{equation*}
\begin{aligned}
w=-(I-B^\dagger B)d,\;\mbox{ where }\;B:=\begin{pmatrix}
&c_i^T,\quad i\in J=I\big(\Proj_\Gamma(u)\big)\\
&a_i^T,\quad i\in R
\end{pmatrix},
\end{aligned}
\end{equation*}
where $B^\dagger$ stands for the Moore-Penrose inverse of the matrix $B$.

Now we examine computing the coderivative of the projector to a {\em specific convex polyhedron}, the intersection of a hyperplane and a box, given by
\begin{equation}\label{eq:feasible}
\begin{aligned}
\Gamma:=&\big\{x\in\R^n\;\big|\;\braket{a,x}=b,\;l\le x \le L\big\}\\
=&\big\{x\in\R^n\;\big|\;a^Tx=b;\,Cx\le d\big\},
\end{aligned}
\end{equation}
where $a\in\R^n$ and $l,L\in\Rb^n$ represent lower and upper bounds (which can be infinite), and where
\begin{equation*}
C:=\begin{pmatrix}
I_n\\
-I_n
\end{pmatrix},
\quad
d=\begin{pmatrix}
L\\
-l
\end{pmatrix}.
\end{equation*}
Without loss of generality, assume that $l_i<L_i$ for all $1\le i \le n$. According to the above analysis, we should take $B:=\begin{pmatrix}
C_{J}\\a^T\end{pmatrix}$,
where $C_J\in\R^{|J|\times n}$ is the matrix formed by the rows corresponding to the index set $J\subset [2n]$ defined by
\begin{equation*}
J:=I\big(\Proj_\Gamma(u)\big)=\big\{i\;\big|\;\Proj_\Gamma(u)_i = L_i\}\cup\big\{n+i\;\big|\;\Proj_\Gamma(u)_i = l_i\big\}.
\end{equation*}

\begin{proposition}[\bf coderivative calculations for specific polyhedra]\label{cod-spec} Let $\Gamma$ be given in \eqref{eq:feasible}, and let the matrix $P$ $($depending on $u\in\R^n)$ be defined by
\begin{equation*}
P:=I-B^\dagger B=\Sigma-\Sigma a(a^T\Sigma a)^\dagger a^T\Sigma=\left\{\begin{aligned}
&\Sigma &\text{ if}\quad \Sigma a =0,\\
&\Sigma-\|\Sigma a\|^{-2}\Sigma a a^T \Sigma &\text{ if}\quad \Sigma a\neq 0,
\end{aligned}\right.
\end{equation*}
where $\Sigma:=I-\Theta$ and where $\Theta:=\mathrm{diag}(\theta)$ with
\[\theta_i:=\left\{
\begin{aligned}
&0\;\textrm{ if}\quad l_i<\Proj_\Gamma(u)_i<L_i,\\
&1\;\textrm{ otherwise}.
\end{aligned}
\right.\]
Then we have the inclusion  $-Pd\subset D^*(-\Proj_\Gamma)(u)(d)$.
\end{proposition}
\begin{proof}
This follows from Theorem~\ref{thm:cod-proj} by the computation of the dagger matrix $B^\dagger$.
\end{proof}

By choosing the regularization matrices $B_k=R:=I-\gamma Q\succ 0$, we get the following Newton system in our algorithm CNFB:
\begin{equation*}
\big((1+\mu_k)I-PR\big)d=\Proj_\Gamma(u^k)-x^k,\;\mbox{ where }\;u^k:=x^k-\gamma(Qx^k+c).
\end{equation*}

Keeping in mind the structure of $P$, the Newton system in CNFB can be solved  as follows.\\[1ex]
$\bullet$ If $\Sigma a=0$, then we have
\begin{equation}
\left\{\begin{aligned}
&(1+\mu_k)d_\delta= (v^k-x^k)_\delta,\\
&(1+\mu_k)I_{\bar\delta}-R_{\bar\delta\times \bar\delta}d_{\bar\delta}=(v^k-x^k)_{\bar\delta}+R_{\bar\delta\times\delta}d_\delta,
\end{aligned}\right.
\end{equation}
where $v^k:=\Proj_\Gamma(u^k)$, $\delta:=\{i\;|\;v_i^k= l_i\;
\text{ or }\;v_i^k= L_i\}$, and $\bar\delta=[n]-\delta=\{i\;|\;l_i<v_i^k<L_i\}$.\\[1ex]
$\bullet$ If $\Sigma a \neq 0$, then we have
\begin{equation}
\left\{\begin{aligned}
&(1+\mu_k)d_\delta=(v^k-x^k)_\delta,\\
&(1+\mu_k)I_{\bar\delta}-P_{\bar\delta\times\bar\delta}R_{\bar\delta\times \bar\delta})d_{\bar\delta}=(v^k-x^k)_{\bar\delta}+P_{\bar\delta\times\bar\delta}R_{\bar\delta\times\delta}d_{\delta}.
\end{aligned}\right.
\end{equation}\vspace*{-0.15in}

\subsection{Numerical Results of CNFB for SLBQP}
Here we compare our algorithm CNFB for SLBQP with the following algorithms:

(1) Gurobi, a commercial software for QPs (which implements a highly optimized IPM);

(2) MATLAB's `quadprog' solver for QPs.

(3) QPPAL \cite{liang2022qppal}, a two-phase proximal augmented Lagrangian method for QPs.

(4) P2GP \cite{di2018two}, a two-phase gradient method specialized for the class of SLBQPs.

The positive semidefinite symmetric matrix \(Q\in\R^{n\times n}\) in our experiments is generated as \(Q=C^TC\), where \(C\in\R^{r\times n}\) ($r \le n$) is a random matrix with i.i.d. Gaussian entries. Thus the rank of $Q$ is expected to be $r$. The results are reported in
Table~\ref{table:SLBQP}. It can be observed that in the high rank case, P2GP is the most efficient algorithm followed by CNFB while Gurobi, MATLAB, QPPAL are less efficient. In the low-rank case, Gurobi has the best performance followed by MATLAB and CNFB, while P2GP and QPPAL are less efficient. Therefore, our proposed CNFB is the only algorithm that is competitive regardless of the rank of the matrix $Q$ in the objective function.\vspace*{-0.15in}

\begin{table}[ht]
\caption{Solving random SLBQPs}\label{table:SLBQP}
\begin{tabular}{llllllll}
\hline
rank&size&measure&CNFB&P2GP&MATLAB&Gurobi&QPPAL\\
\hline
\multirow{6}{*}{\(r=n\)}&\multirow{2}{*}{1000}&time &0.31s &0.24s &0.79s &0.72s &5.30s\\
&&residual &7.38e-13 &4.16e-09 &3.17e-08 &2.94e-06 &3.79e-06\\
&\multirow{2}{*}{2000}&time &0.91s &0.28s &4.70s &3.16s &16.06s\\
&&residual &2.25e-12 &9.86e-09 &6.12e-07 &2.39e-05 &7.11e-07\\
&\multirow{2}{*}{5000}&time &7.67s &1.55s &47.92s &39.98s &208.36s\\
&&residual &7.86e-12 &3.68e-10 &1.01e-06 &1.46e-04 &1.91e-04\\
\hline
\multirow{6}{*}{\(r=0.9n\)}&\multirow{2}{*}{1000}&time &0.26s &0.21s &0.78s &0.67s &5.64s\\
&&residual &8.39e-13 &5.43e-09 &7.40e-10 &8.54e-05 &1.27e-06\\
&\multirow{2}{*}{2000}&time &0.99s &0.35s &4.67s &3.18s &16.13s\\
&&residual &1.93e-12 &9.08e-09 &4.43e-06 &9.28e-05 &1.20e-06\\
&\multirow{2}{*}{5000}&time &9.20s &1.44s &50.11s &41.90s &424.25s\\
&&residual &7.90e-12 &3.62e-08 &1.30e-05 &8.97e-05 &1.73e-05\\
\hline
\multirow{6}{*}{\(r=0.5n\)}&\multirow{2}{*}{1000}&time &0.84s &0.55s &0.73s &0.43s &5.10s\\
&&residual &6.85e-13 &2.09e-09 &4.17e-06 &3.34e-05 &1.60e-05\\
&\multirow{2}{*}{2000}&time &4.56s &2.07s &4.18s &1.53s &24.89s\\
&&residual &1.92e-12 &8.58e-09 &1.64e-07 &7.23e-06 &3.32e-05\\
&\multirow{2}{*}{5000}&time &36.42s &12.80s &39.57s &13.40s &1010.14s\\
&&residual &6.34e-12 &2.58e-08 &1.45e-05 &1.42e-04 &1.97e-05\\
\hline
\multirow{6}{*}{\(r=0.1n\)}&\multirow{2}{*}{1000}&time &2.75s &17.10s &0.71s &0.16s &6.76s\\
&&residual &1.69e-09 &1.19e+00 &3.10e-09 &2.11e-11 &2.91e-04\\
&\multirow{2}{*}{2000}&time &8.76s &27.08s &5.20s &0.47s &70.24s\\
&&residual &4.22e-08 &2.77e+00 &2.70e-06 &2.43e-08 &1.58e-05\\
&\multirow{2}{*}{5000}&time &52.85s &120.59s &51.60s &3.07s &1710.12s\\
&&residual &1.18e-08 &1.86e-03 &1.22e-05 &2.13e-05 &3.63e-05\\
\hline
\end{tabular}
\end{table}

\section{Applications of CNAL to Lasso Problems}\label{sec_numerical}\vspace*{-0.05in}

This section presents some implementations of the proposed CNAL to the class of Lasso problems. We confine our attention to the basic Lasso problem introduced in \cite{tibshirani1996regression} as the \(\ell_1\)-{\em regularized least squares regression model} \cite{tibshirani1996regression}:
\begin{equation}\label{lasso}
\mbox{minimize }\;\vp(x):=\frac{1}{2}\|Ax-b\|^2_2+\lambda\|x\|_1\;\mbox{ over }\;x\in\R^n,
\end{equation}
where \(A\in\mathbb{R}^{m\times n}\) is the data matrix (with \(m\)  being the number of samples and \(n\)  being the number of features), \(b\in\mathbb{R}^m\), \(\lambda>0\), and \(\|\cdot\|_1\), \(\|\cdot\|_2\) are the \(\ell_1\)-norm and \(\ell_2\)-norm, respectively. We see that the Lasso problem \eqref{lasso} is a special case of \eqref{problem:primal} with $h(Ax)=\frac{1}{2}\|Ax\|^2$, $c=A^*b$, $p(x)=\lambda\|x\|_1$, $h^*(y)=\frac{1}{2}\|y\|^2$, and $p^*(z)=\delta_{B_{\infty}(\lambda)}$, where $B_{\infty}(\lambda):=\{z\in\R^n\;|\;\|z\|_{\infty}\le \lambda\}$ is the closed ball with radius $\lambda$ in the $\ell_\infty$-norm.

The proximal mapping of \(\sigma p\) is the {\em soft-thresholding operator}
\begin{equation*}
[\Prox_{\sigma p}(u)]_i=\mathrm{sign}(u_i)(|u_i|-\sigma\lambda)_{+} \,\,\,\text{for}\,\,\, i=1,\ldots,n,
\end{equation*}
where  \((|u_i|-\sigma)_{+}:=\max\{0,|u_i|-\sigma\}\).

In order to apply CNAL to the Lasso problem \eqref{lasso}, we need to compute the coderivative of $\Prox_{\sigma p}$. Since \(\Prox_{\sigma p}\) is separable and piecewise linear, its coderivative is easily calculated by
\begin{equation*}
\big(D^*\Prox_{\sigma p})\big)(u)(d)=\big\{w\in\R^n\;\big|\;w_i\in G_i(u_i,d_i)\big\},
\end{equation*}
where the set-valued mapping \(G:\R\times\R\rightrightarrows\R\) is defined by
\begin{equation*}
G_i(u_i,d_i)=\left\{\begin{aligned}
&d_i&\text{if\quad}&|u_i|>\sigma\lambda,\\
&\{0,d_i\}&\text{if\quad}&u_i=\sigma\lambda, d_i> 0,\\
&[0,d_i]&\text{if\quad}&u_i=\sigma\lambda, d_i\le0,\\
&0 &\text{if\quad}&|u_i|<\sigma\lambda,\\
&[0,d_i]&\text{if\quad}&u_i=-\sigma\lambda, d_i\ge 0,\\
&\{0,d_i\}&\text{if\quad}&u_i=-\sigma\lambda, d_i< 0.\\
\end{aligned}
\right.
\end{equation*}
We can directly extract a linear mapping from $(D^*\Prox_{\sigma p})(u)$ as follows: for all \(d\in\R^n\), take \(P d\subset D^*(\Prox_{\sigma p})(u)(d)\), where \(P\in\R^{n\times n}\) is a diagonal matrix with the entries
\begin{equation*}
P_{ii}:=\left\{\begin{aligned}
&1 &\text{\quad if \quad} &|u_i|>\sigma\lambda,\\
&0 &\text{\quad if \quad} &|u_i|\le\sigma\lambda.
\end{aligned}
\right.
\end{equation*}
Note that for the Lasso problem under consideration, the proposed CNAL is closely related to SSNAL designed to solve Lasso models in \cite{li2018efficiently}, because the Newton systems in both algorithms are identical in this case with the major difference in the linesearch strategy. CNAL uses the Wolfe
linesearch for the coderivative Newton subproblem solver GRNM-W while SSNAL employs the backtracking Armijo linesearch for the semismooth Newton subproblem solver as stated in \cite{li2018efficiently}. Let us point out that the actual MATLAB suite (called SuiteLasso) developed in \cite{li2018efficiently} uses a kind of Wolfe linesearch for SSNAL instead of the Armijo linesearch. In the numerical implementations, our codes are adapted from SuiteLasso, with the major change being a {\em different Wolfe linesearch} that appears to be more efficient, at least for random instances.\vspace*{-0.1in}

\subsection{Numerical Results of CNAL for Lasso}
Here we present the results of numerical experiments to compare our algorithm CNAL with the following well-known first-order and second-order algorithms to solve Lasso problems:

(1) The {\em Semismooth Newton Augmented Lagrangian Method} (SSNAL)\footnote{\url{https://github.com/MatOpt/SuiteLasso}} \cite{li2018highly}, one of the most efficient methods for Lasso.

(2) The {\em Matrix-Free Interior Point Method} (mfIPM)\footnote{\url{http://www.maths.ed.ac.uk/ERGO/mfipmcs/}} \cite{fountoulakis2014matrix}, an interior point method that is highly optimized for Lasso.

(3) The (Nesterov) {\em Accelerated Proximal Gradient Method} (APG) \cite{nesterov1983method} and the closely related in this case {\em Fast Iterative Shrinkage-Thresholing Algorithm} (FISTA)\footnote{We use the implementation in SLEP: \url{http://yelabs.net/software/SLEP/}} \cite{beck2009fast}, a simple and efficient first-order method.

(4) The {\em Alternating Direction Method of Multipliers} (ADMM)\footnote{We use the implementation at \url{https://web.stanford.edu/~boyd/papers/admm/lasso/lasso.html}} \cite{gabay1976dual,glowinski1975approximation}, a classical and popular primal-dual splitting method.

We tested these algorithms on the standard data set LIBSVM \cite{chang2011libsvm}. The parameter $\lambda$ in \eqref{lasso} is chosen to be $\lambda=\lambda_c\|A^Tb\|_\infty$ where $\lambda_c=10^{-3}$. The relative KKT residual
$$
\eta=\frac{\|\tilde x-\mathrm{Prox}_{\lambda\|\cdot\|_1}(\tilde x-A^T(A\tilde x-b))\|}{1+\|\tilde{x}\|+\|A^T(A\tilde{x}-b)\|}
$$
is used to measure the accuracy of an approximate minimizer $\tilde x$ of \eqref{lasso}. The results are reported in Table~\ref{table_lasso}. In the table, `NA' means `not applicable' (it appears because mfIPM is not applicable by design in the $m>n$ case). We can observe that CNAL is as efficient as SSNAL and both algorithms are highly accurate and much faster than the others.

\begin{table}[ht]
\caption{Solving Lasso problems in LIBSVM}\label{table_lasso}
\begin{tabular}{lllllll}
\hline
Problem&Measure&CNAL&SSNAL&mfIPM&APG&ADMM\\
% \multicolumn{5}{l}{Algorithms}\\
% \cmidrule(lr){4-*}
% &\\
\hline
\multirow{1}{*}{covtype}&time&0.01s&0.01s&NA&70.49s&11.95s\\
\multirow{1}{*}{581012;54}&residual&2.17e-07&2.18e-07&NA&2.75e-05&4.26e-04\\
\hline
\multirow{1}{*}{YearPredictionMSD}&time&0.01s&0.01s&NA&326.84s&55.99s\\
\multirow{1}{*}{463715;90}&residual&2.48e-07&3.74e-07&NA&5.66e-04&6.18e-04\\
\hline
E2006.test &time    &0.14s    &0.14s    &2.04s    &0.10s    &21.85s\\
3308;72812 &residual&8.25e-07 &1.60e-07 &4.62e-09 &2.97e-07 &4.68e-07\\
\hline
\multirow{1}{*}{E2006.train}&time&0.40s&0.38s&5.84s&0.36s&187.12s\\
\multirow{1}{*}{16087;150348}&residual&8.51e-07&1.65e-07&5.65e-09&2.94e-06&8.54e-08\\
\hline
news20\underline{ }tfidf\underline{ }test&time&0.34s&0.33s&14.71s&21.02s&60.41s\\
7532;49909&residual&5.02e-07&6.77e-07&1.69e-05&3.01e-06&8.03e-07\\
\hline
housing\underline{ }expanded7 &time     &1.91s    &1.93s    &249.77s    &114.20s    &104.52s\\
506;77520                     &residual &8.83e-07 &8.83e-07 &2.97e-01   &7.57e-04   &1.78e-04\\
\hline
pyrim\underline{ }expanded5   &time    &1.14s    &1.15s      &875.85   &63.90s   &156.27s\\
74;169911                     &residual&9.37e-07 &9.37e-07   &4.46e-07 &2.57e-03 &3.50e-04\\
\hline
\end{tabular}
\end{table}\vspace*{-0.1in}

\section{Conclusions}\label{sec_conclusions}\vspace*{-0.05in}

This paper proposes the globally convergent coderivative-based generalized regularized Newton method with the Wolfe linesearch (GRNM-W) to solve \(C^{1,1}\) optimization problems. The local convergence rate of GRNM-W is at least linear and becomes superlinear under the semismooth$^*$ property of the gradient mapping. We also presented a modified version of GRNM-W that is applicable to arbitrary nonconvex functions. Under the imposed KL properties of the objective function, we proved convergence and convergence rate results for the modified GRNM-W. We further combined the coderivative-based Newton method GRNM-W with the
forward-backward envelope and the augmented Lagrangian method, proposing the two algorithms CNFB and CNAL for solving convex composite minimization problems that are first-order nonsmooth and constrained. Numerical experiments on support vector machines (formulated as special quadratic programming problems) and the Lasso problem indicate that both CNFB and CNAL are efficient, which confirms the effectiveness of the main novel algorithmic scheme GRNM-W. Our future research includes further applications of the proposed nonsmooth Newton algorithms to a large variety of optimization problems arising in machine learning, data science, statistics, and other fields.\vspace*{0.05in}

{\bf Acknowledgments} Acknowledgments are not compulsory. Grant or contribution numbers may be acknowledged.
Please refer to journal-level guidance for any specific requirements.\vspace*{0.05in}

{\bf Data availability} The data that supports the findings of this study is available from the corresponding author upon request.\vspace*{0.05in}

{\bf Code availability} The code that supports the findings of this study is available from the corresponding author upon request.\vspace*{-0.2in}

\section*{Declarations}\vspace*{-0.05in}
{\bf Conflict of interest} The authors declare that they have no conflict of interest.

\end{document}